%% file: epsilon_main.tex
\documentclass{amsart}
\usepackage{amsmath,amsthm,amsfonts,amssymb, mathrsfs}
\usepackage{graphicx,xcolor}
\usepackage{comment}
\usepackage{todonotes}[]
\usepackage{tikz}

\usetikzlibrary{automata, positioning, arrows.meta}

\usepackage[style=alphabetic]{biblatex}
\newcommand{\RR}{\mathbb{R}}

\newcommand{\ZZ}{\mathbb{Z}}

\newcommand{\HH}{\mathbb{H}}
\newcommand{\PP}{\mathbb{P}}

\DeclareMathOperator{\End}{End}
\DeclareMathOperator{\trace}{tr}

\DeclareMathOperator{\diam}{diam}

\newcommand{\norm}[1]{\left\lVert#1\right\rVert}

\newtheorem{thm}{Theorem}[section]
\newtheorem{lemma}[thm]{Lemma}
\newtheorem{prop}[thm]{Proposition}
\newtheorem{coro}[thm]{Corollary}

\theoremstyle{remark}
\newtheorem{rmk}[thm]{Remark}

\theoremstyle{definition}
\newtheorem{defi}[thm]{Definition}
\newtheorem{ex}[thm]{Example}
\newtheorem{lem}[thm]{Lemma}
\newtheorem*{defi*}{Definition}

\title{On separated families of Anosov representations}

\author{Joaqu\'in Lejtreger}
\address[Joaqu\'in Lejtreger]{
Sorbonne Université and Université Paris Cité, CNRS, IMJ-PRG, F-75005 Paris, France.}
\email{lejtreger@imj-prg.fr}
\date{}

\author{Joaqu\'in Lema}
\address[Joaqu\'in Lema]{Department of Mathematics, Boston College, Chestnut Hill, MA 02467}
\email{lemajo@bc.edu}
\date{}

\begin{document}

\begin{abstract}
We introduce different notions of separation for families of Anosov representations. We show that, along a diverging sequence of such families, the critical exponent is asymptotic to a combinatorial invariant computable from the spectral data of a finite graph. Our method allows us to derive bounds on the Thurston asymmetric metric of~\cite{carvajales2024thurston}. As an application, we study specific degenerations of convex projective structures on a pair of pants, generalizing an example of McMullen.
\end{abstract}

\maketitle

\input{sections/intro}

\input{sections/epsilon} 
\input{sections/thermodynamic}
\input{sections/proof}
\input{sections/sl3}

\appendix
\input{sections/finite_type.tex}

\printbibliography

\end{document}

%% file: sections/intro.tex
\section{Introduction}
The goal of this article is to estimate the critical exponent of diverging families of representations of a hyperbolic group $\Gamma$ into a semisimple Lie group $G$, under a dynamical condition on a chosen generating set that we call \emph{strong separation}.

Our motivation comes from an example of McMullen \cite{mcmullen1998hausdorff}, recently reviewed and expanded by Courtois-Guilloux \cite{guilloux2024hausdorff} and Dang-Mehmeti \cite{nguyen-bac2024variation}, who compute asymptotics for the critical exponent for diverging Schottky representations acting on $\HH^n$.

Our framework unifies and generalizes these computations to higher rank.
Furthermore, our techniques also provide control along strongly separated sequences of the Thurston asymmetric metric introduced in \cite{carvajales2024thurston}.

As an application, we study degenerations of convex projective structures on a pair of pants, whose associated holonomies define strongly separated representations of the free group into $SL_3(\RR)$.

\subsection{The setup}

Given $\Gamma$ a hyperbolic group, $G$ a semisimple Lie group with non-compact factors, $\mathfrak{a}^+$ a Weyl chamber, and $\Theta$ a set of simple roots, we will be concerned with the study of $\Theta$-Anosov representations. This notion was introduced in \cite{labourie2004anosov} and \cite{guichard2012anosov}, and can be thought of as a higher-rank analog of convex cocompact representations.
Such representations come equipped with a pair of continuous, dynamics-preserving, and equivariant maps 
$$(\xi, \xi^{op}): \partial \Gamma \to \mathcal{F}_\Theta \times \mathcal{F}^{op}_\Theta,$$
where $\mathcal{F}_\Theta$ and $\mathcal{F}^{op}_\Theta$ are the opposite partial flag manifolds associated to $\Theta$. 

We consider the \emph{Cartan projection} $\kappa: G \to \mathfrak{a}^+$, and the \emph{Jordan projection} $\lambda: G \rightarrow \mathfrak{a}^+$. In $SL_n(\RR)$, the Cartan projection is the vector containing the log of the singular values of a matrix in order, and the Jordan projection is the vector containing the log of the moduli of the eigenvalues of a matrix in order.

For a representation $\rho: \Gamma \to G$, we define the \emph{limit cone} $\mathcal{L}_\rho$ of $\rho$ as the smallest closed convex cone of $\mathfrak{a}^+$ containing $\lambda (\rho(\Gamma))$. 

There are many length functions associated to a given Anosov representation, each determined by a functional $\varphi \in \mathfrak{a}^*$ and corresponding to a different notion of growth in the symmetric space: we will say that $\varphi$ is \emph{positive in the limit cone of $\rho$} if $\varphi|_{\mathcal{L}_\rho - \{0\}} > 0$. 
Given such a functional, we consider its \emph{length function} $L^\varphi_\rho : \Gamma \rightarrow \mathbb{R}_{>0}$ via
$$L^\varphi_\rho (\gamma) = \varphi (\lambda (\rho (\gamma))).$$
Each such length function has an associated \emph{critical exponent}, defined via
$$h_\rho(\varphi) := \limsup_{t\to \infty} \frac{1}{t} \log \# \{ [\gamma] \in [\Gamma] : L^\varphi_\rho (\gamma) \leq t\} \in [0,\infty],$$
where $[\Gamma]$ is the set of conjugacy classes of $\Gamma$.

To estimate these critical exponents, it will be necessary to code the limit map and apply techniques from thermodynamic formalism. To do so, we use the notion of \emph{strong Markov codings}, defined by Ghys-de la Harpe in \cite{ghys1990sur}, based on the work of Cannon \cite{cannon1984combinatorial}, and proved to be a successful tool in the ergodic theory of hyperbolic group in the work of Cantrell-Tanaka \cite{cantrell2025manhattan}.

Given a symmetric set of generators $S$ for $\Gamma$, a strong Markov coding is a finite directed graph $\mathcal{G} = (\mathcal{V},\mathcal{E})$ equipped with a labeling of the edges $\pi: \mathcal{E} \to S$ by elements of $S$. Such labels allow us to associate with any finite path a geodesic segment in $\Gamma$, and similarly, to any bi-infinite path, a geodesic in the group.
Collecting the labels traversed by a bi-infinite path in $\mathcal{G}$, we obtain a dynamical system $X_\mathcal{G} \subset S^\mathbb{Z}$ equipped with the shift map, called a \emph{sofic shift}.
Such a sofic shift codes the boundary via a map $p: X_\mathcal{G} \rightarrow \partial \Gamma$, sending the sequence of labels of a bi-infinite geodesic path to its forward endpoint on $\partial\Gamma$.

Using this language, the work of Bochi-Potrie-Sambarino \cite{bochi2019anosov} provides a criterion for detecting Anosov representations adapted to a given strong Markov coding. 
More precisely, one can prove that a representation $\rho: \Gamma \rightarrow G$ is $\Theta$-Anosov if and only if we can assign for every vertex $v \in \mathcal{V}$, an open set $M_v$ in the partial flag $\mathcal{F}_\Theta$, called a multicone, in such a way that, for each $e = (v_0, v_1) \in \mathcal{E}$,
$$\rho(\pi(e))^{-1}M_{v_1} \subseteq M_{v_0}.$$
Such multicones serve as a ``coarse'' Markov partition of the limit set.
One can show that $\xi \circ p (\underline{x}) \in M_{v_0}$, for every $\underline{x} \in X_\mathcal{G}$ induced by a path starting at $v_0$ (see Theorem \ref{thm:coding}).

Every element in the image of a $\Theta$-Anosov representation acts in the partial manifold $\mathcal{F}_\Theta$ with north-south dynamics, in the sense that for every $\gamma \in \Gamma$ of infinite order, we get the property that $\rho (\gamma)^n$ acting on $\mathcal{F}_\Theta$ converges uniformly in compact sets to $\xi (\gamma_+)$ outside a compact set $C_{\rho (\gamma)}^-$, called the Schubert locus.
The separation conditions we impose require these north-south dynamics of the generators to be aligned with the multicone partition.

Informally, $\rho$ is $\varepsilon$-separated with respect to $\mathcal{G}$ and $\Theta$ if we can choose an invariant multicone family $\{M_v\}_{v \in \mathcal{V}}$ in $\mathcal{F}_\Theta$ in such a way that $M_v$ is $\varepsilon$-away from the Schubert loci of the repelling flags of generators corresponding to incoming edges. It is $\varepsilon$-strongly separated if, in addition, each generator $\rho(s)$ exhibits uniform north-south behavior in $\mathcal{F}_\Theta$: its attracting and repelling flags are $2\varepsilon$-apart, and its action contracts the $\varepsilon$-complement of the repelling Schubert locus into the $\varepsilon$-ball around the attracting flag. See Definition~\ref{def:separatednested} for a detailed definition.

Notice that the strong separation condition resembles the usual definition of a ``ping-pong'' system, except that admissible words are governed by the strong Markov coding, taking into account relations in $\Gamma$.

To verify that a representation is separated, we only need to check finitely many conditions, which makes this notion easy to verify in specific examples, as we will see, for instance, in the proof of Theorem \ref{thm:sl3}.

The key remark is that, for a given $\varepsilon > 0$, a representation being $\varepsilon$-strongly separated allows us to identify regions of the limit set where the Busemann cocycle is well-behaved.
For example, in rank one, these conditions imply the existence of $C>0$ (only depending on $\varepsilon$) such that for any convex-cocompact representation $\rho$ that is strongly separated, we get the following lower bound for the Busemann function:
$$d(o,\rho(g)^{-1}o) - C \leq b_x (o,\rho(g)^{-1}o),$$
where $S$ denotes the set of generators, $g, h \in S$, and $x$ belongs to the subset of the limit set codified by infinite words finishing with $h$, provided $gh$ is part of the language of the Markov coding.
See Lemma \ref{lemma:key} for the higher rank version of this inequality.

\subsection{Main results}
Having established the setup, we proceed to state the main results.
Given a choice of a Weyl chamber $\mathfrak{a}^+$, we denote the space
$$\mathfrak{a}_\Theta = \bigcap_{\alpha \in \Pi \setminus \Theta} \ker \alpha.$$

Two facts drive the argument of our main theorems: 
first, a characterization of $h_\rho (\varphi)$ as the unique zero of the function $s \to P(- s \varphi \circ B)$, where $B$ is an $\mathfrak{a}_\Theta$-valued potential obtained from the Iwasawa cocycle (see Proposition \ref{prop:poincareradioespectral}).
Second, the strong separation condition bounds $\varphi \circ B$ between two locally constant potentials depending solely on the Cartan projection on our given set of generators $S$.

We remark that the first item holds regardless of any separation condition and already appears in the literature in different forms. 
Our approach is closer to Quint \cite{quint2003indicateur}, although it appears in some form in work by Ledrappier \cite{ledrappier1994structure} and Sambarino \cite{sambarino2014hyperconvex}.

The second ingredient suggests that, under strong separation, we can produce the following approximation for $h_\rho (\varphi)$:

\begin{defi*}
 Given $\rho: \Gamma \rightarrow G$ a $\Theta-$Anosov representation, $\mathcal{G}$ a strong Markov structure and $\varphi \in (\mathfrak{a}_\Theta)^*$ positive on the limit cone of $\rho$, we define the \emph{approximating} Cartan exponent $h_\rho^\kappa (\varphi)$ as the number defined via
    $$P(- h_\rho^\kappa (\varphi) \varphi \circ C) =0,$$
 where $C : X_\mathcal{G} \rightarrow \mathbb{R}$ is the vector-valued potential $C(\underline{x}) = \kappa (x_0^{-1})$, and $P$ denotes the pressure on the shift.
\end{defi*}

Since the potential $C$ depends only on the first coordinate, the number $h_\rho^\kappa (\varphi)$ can be found by computing the spectral radius of a matrix depending on the combinatorics of $\mathcal{G}$, making this quantity easy to compute (see Appendix \ref{app:subshifts}).
Putting the two previously mentioned facts together, we obtain the following theorem:

\begin{thm}
 Given $\varepsilon>0$, let $(\rho_n) \in Hom (\Gamma,G)^\mathbb{N}$ be a sequence of $\varepsilon-$strongly separated representations with respect to $\Theta$ and $\mathcal{G}$ such that $\min_{g \in S} \varphi (\kappa (\rho_n (g))) \to \infty$, then
    $$\lim_{n \to \infty} \frac{h_{\rho_n} (\varphi)}{h_{\rho_n}^\kappa (\varphi)} = 1.$$
\end{thm}

We can, in fact, prove a slightly better statement, as one can check in Theorem \ref{thm:general+}.
Moreover, in Section \ref{sec:proof} we also provide some weaker estimates that assume only separation (not strong separation).

We can show that regardless of our use of a Markov coding, a group admitting a sequence of representations as in the previous theorem is necessarily (virtually) free (see Proposition \ref{prop:onlyfree}).
Although this limits the applicability of our result, it is important to note that the choice of automaton is crucial for obtaining precise asymptotics.
Particularly, it is not always true that any sequence of representations of the free group is strongly separated with respect to the standard set of generators.

Under this strong separation hypothesis, one also gets asymptotic control on the length functions, providing good control on the limit cone along such degenerations (see Proposition \ref{coro:limitcones}). 
This also provides information about the Thurston asymmetric metric along these sequences.
Given two representations $\rho_1,\rho_2$, and $\varphi$, a functional positive in the limit cone of both representations, the Thurston asymmetric metric is defined by
$$d_{Th}^\varphi (\rho_1,\rho_2) = \log \sup_{[\gamma] \in [\Gamma]} \left( \frac{h_{\rho_2}^\varphi L^\varphi_{\rho_2} ([\gamma])}{h_{\rho_1}^\varphi L^\varphi_{\rho_1} ([\gamma])} \right).$$
This notion was introduced in \cite{carvajales2024thurston}, where it was proven to define an asymmetric metric under the assumption that any automorphism of $\mathfrak{a}$ leaving $\varphi$ invariant is inner (and studied in greater generality in \cite{jyothis2026horoboundary}).
On this subject, we can prove the following theorem:

\begin{thm}
Let $\Theta \subset \Pi$ be invariant under the opposition involution, and $\varphi \in \mathfrak{a}_\Theta^*$ for which $d^{\varphi}_{Th}$ defines an asymmetric metric on representations for which $\varphi$ is positive in its limit cone. 
Then, given $(\rho_n) \in Hom (\Gamma,G)^\mathbb{N}$ a sequence of $\varepsilon$-strongly separated representations with respect to $\Theta$ and $\mathcal{G}$ for which $\min_{g \in S} \varphi (\kappa (\rho_n (g)))\to \infty$, and $\max_{g \in S} \varphi (\kappa (\rho_n (g)))/\min_{g \in S} \varphi (\kappa (\rho_n (g)))$ is bounded above and below uniformly in $n$ for any $g \in S$, then:
$$\sup_n \max\{ d_{Th}^\varphi (\rho_n, \rho_0), d_{Th}^\varphi (\rho_0,\rho_n)\} <\infty.$$
\end{thm}

As an application of this work, we study degenerations of convex projective structures on a pair of pants, originally parametrized by Goldman in \cite{goldman1990convex} as an $8$-dimensional open cell.
We will use Fock-Goncharov's parametrization (see \cite{fock2007moduli}), consisting on two triple-ratios $(X_1, X_2)$, and six cross-ratios $\vec{Z},\vec{W} \in \mathbb{R}^3_{>0}$ (see Section \ref{sec:sl3} for more details).

Of all these parameters, it is particularly interesting to understand how the geometry varies along families with different triple ratios, as these are, in some sense, the truly higher-rank parameters in these representations.
Denote by $\mathcal{RP} (X_1, X_2)$ the slice of convex projective structures with constant triple ratio parameters $X_1, X_2$.

Denote by $\mathfrak{a}$ Cartan subspace of diagonal matrices in $\mathfrak{sl}_3 (\mathbb{R})$, and $\omega_i$ the vectors defined via $\alpha_j (\omega_i) = \delta_{ij}$, for $i,j =1,2$.

\begin{thm}
 Let $\varphi \in \mathfrak{a}^*$ and $(\vec{Z_t})_{t \in \mathbb{R}_{\geq 0}}, (\vec{W_t})_{t \in \mathbb{R}_{\geq 0}}$ be a path of vectors such that $\vec{Z_t},\vec{W_t} \to \infty$.
 Then, given $(\rho_t) \in \mathcal{RP}(X_1, X_2)^{\mathbb{R}_{\geq 0}}$ a path of holonomies of convex projective structures with edge invariants $\vec{W_t},\vec{Z_t}$, and assuming $\varphi$ is positive on their limit cone, we get that
   \begin{enumerate}
      \item Let $v_i (t) = (\log (W^i_t Z^{i-1}_t)) \omega_1 + \log (Z^i_t W^{i-1}_t) \omega_2$, $w_i (t) = (\log (Z^i_t W^{i-1}_t)) \omega_1 + (\log (W^i_t Z^{i-1}_t)) \omega_2$, we get that for any $\eta>0$ there is $T>0$ for which
      $$\frac{2 \log 2}{\max_i \{ \varphi (v_i (t)),\varphi (w_i (t)) \}} - \eta \leq h_{\rho_t} (\varphi) \leq \frac{2 \log 2}{\min_i \{ \varphi (v_i (t)),\varphi (w_i (t)) \}} + \eta,$$
 for every $t > T$.
      \item Let $\rho_t \in \mathcal{RP} (X_1, X_2)$, and $\rho_t' \in \mathcal{RP} (X_1', X_2')$ be a path of representations in the constant triple ratio slice, and defined by the same family of cross ratios given by $\vec{Z_t},\vec{W_t}$, then
      $$d_{Th}^\varphi (\rho_t,\rho_t') \to_{t \to \infty} 0.$$
      \item If $\max \{\varphi (v_i (t)),\varphi (w_i (t))\}_{i=1,2,3}/\min \{\varphi (v_i (t)),\varphi (w_i (t))\}_{i=1,2,3}$ is bounded above and below for every $t$, then the Thurston asymmetric metric defined by $\varphi$ is finite.
   \end{enumerate}
\end{thm}

One can verify that McMullen's example belongs to the family of examples described above. In that case, the first item is recovering the precise asymptotic to the critical exponent in this setup.

Notice that the second item represents a new, higher-rank phenomenon, as this cannot happen for convex-cocompact hyperbolic structures on the pair of pants.
This is a consequence, for instance, of the spinal metric ribbon graph construction by Bowditch and Epstein (see \cite{bowditch1988natural}).
It can also be deduced from a direct application of our results.

The classical construction of the spinal ribbon graph also provides a geometric interpretation of the third item: the asymmetric metric is incomplete along directions where cuff lengths grow at the same rate. The metric graph obtained as the geometric limit of the normalized pants (so that they have constant area) lies in the completion of the space. This point of view is explored Courtois-Guilloux in \cite{guilloux2024hausdorff}, and Xu \cite{xu2019incompleteness} for the study of the pressure metric.

In the setup of convex projective structures, Parreau \cite{parreau2015invariant} associates to certain families of degenerating structures a piecewise Euclidean surface in a building, which provides asymptotic estimates for the lengths of curves.

In Theorem \ref{thm:divergingpants}, we are able to detect families of representations (that we call shear diverging) that are strongly separated, whose parameters satisfy conditions similar to the ones stated by Parreau.

An interesting direction could be to recast our result in terms of graph-like objects in the symmetric space, aiming at a result similar to Parreau's.


\subsection{Organization of the paper}

 Section \ref{sec:epsilon} reviews linear algebraic estimates that are needed for the proof of the main theorem. Along the way, we establish the necessary background on symmetric spaces. 

Section \ref{sec:thermodynamic} adapts the work of \cite{bochi2019anosov} to the setup of strong Markov partitions, and we show that we can compute the critical exponents from the Busemann potential.

 Section \ref{sec:proof} proves the main result and some variations that can be applied in different contexts.

Section \ref{sec:sl3} explains the Fock-Goncharov parametrization of convex projective structures on the pair of pants and uses the estimates of Section \ref{sec:proof} to obtain asymptotics for the critical exponents of some explicit diverging sequences of representations.
 Appendix \ref{app:subshifts} explains some basic theory on subshifts of finite type, which we need for explicit computations in Section \ref{sec:sl3}.

\subsection*{Acknowledgments}
We are grateful to Gilles Courtois for introducing us to McMullen's example and showing us his perspective using thermodynamic formalism. We are grateful to Giuseppe Martone and Le\'{o}n Carvajales for useful comments. We also thank Martin Bridgeman and Andr\'es Sambarino for their guidance.

Joaqu\'in Lejtreger has received funding from the European Union’s Horizon 2020 research and innovation program under the Marie Sk\l{}odowska-Curie grant agreement No.\ 945332. Joaqu\'in Lema finished this article while in residence at the SLMath Institute in Berkeley, California, supported by the National Science Foundation Grant No. DMS-2424139. The authors acknowledge support of the Institut Henri Poincaré (UAR
839 CNRS-Sorbonne Université), and LabEx CARMIN (ANR-10-LABX-59-01)

%% file: sections/epsilon.tex
\section{North-south dynamics in flag manifolds}
\label{sec:epsilon}

In this section, we introduce the notion of $\varepsilon$-separated $\Theta-$loxodromic elements in $G$ and prove an estimate which will be fundamental to our main result (Proposition \ref{coro:estimatesgeneralcase}).
Along the way, we review some Lie theoretic background/notations that we will use in this paper.

\subsection{Proximal elements}
\label{subsec:proximal}
We assume $V$ is a finite-dimensional, real vector space.
A linear endomorphism $g \in \End (V)$ is said to be \emph{proximal} if it has a simple eigenvalue of maximal modulus, that we denote as $\lambda_1 (g)$.
We also denote by $g_+ \in \mathbb{P} (V)$ the eigenspace associated to $\lambda_1 (g)$, and by $g_- \in \mathbb{P} (V^*)$ the unique $g-$invariant hyperplane complementary to $g_+$.

Recall that an inner product $\langle \cdot,\cdot \rangle$ on $V$ induces a metric in $\mathbb{P} (V)$ by declaring that $\cos (d(p_1,p_2)) = \vert \langle v_{p_1},v_{p_2}\rangle \vert$, where $v_{p_i} \in p_i$ have unit norm.
We will fix a choice of inner product for the remainder of the section.
Given $g \in \End (V)$ proximal, we will denote the open $\varepsilon$-ball around $g_+$ as $\mathcal{U}^+ (g,\varepsilon)$, and by $\mathcal{U}^- (g,\varepsilon)$ the complement of the $\varepsilon$-open neighborhood of $g_-$ in $\mathbb{P} (V)$.

\begin{defi}
 We will say that a proximal $g \in \End (V)$ is \emph{$\varepsilon$-separated} if
    \begin{itemize}
        \item $d (g_+,g_-) > 2 \varepsilon.$
        \item $g (\mathcal{U}^- (g,\varepsilon)) \subset \mathcal{U}^+ (g,\varepsilon)$.
    \end{itemize}
\end{defi}

\begin{rmk}
 This is a slight weakening of the notion of $\varepsilon$-proximality introduced by Benoist in \cite{benoist_actions_1996}.
 Many of the estimates we include here already appear in Benoist's work in some form, with the main difference being that we want effective constants in our estimates.
\end{rmk}

Denoting by $*$ the adjoint induced by our choice of inner product, for every $g \in \End(V)$ the transformation $g^* g$ is self-adjoint, and therefore, diagonalizable in an orthonormal basis with non-negative eigenvalues.
Their square-roots, ordered as $\sigma_1 (g) \geq \sigma_2 (g) \geq \ldots$ are called the \emph{singular values} of $g$.
The eigenvector associated with $\sigma_i^2 (g)$ is called the (right) $i-$th singular vector.

We will denote as $s (g)$ a first singular vector, and $S(g)$ its orthogonal complement (spanned by the remaining singular vectors).
Notice that the first singular value $\sigma_1 (g)$ is simply $\norm{g} = \sup_{v \in V \setminus \{0\}} \frac{\norm{gv}}{\norm{v}}$.

We will soon mention that singular values are related to distances in symmetric spaces, which often differ from the eigenvalue data of $g$.
The following is a classical lemma of linear algebra, the proof can be found, for example, in the appendix A of \cite{bochi2019anosov}.

\begin{lem}
\label{lemma:bps}
 Let $g \in \End (V)$, then for every $v \in V$ we have
    $$\|gv\| \ge \|g\|\|v\|\sin(\angle v, S(g)).$$
\end{lem}

We get the following crucial consequence:

\begin{lemma}
\label{lemma:proximalestimate}
 Given $\varepsilon >0$, any $\varepsilon$-separated proximal transformation $g$ verifies
 $$\sin (\varepsilon) \norm{g} \leq \lambda_1 (g) \leq \norm{g},$$
 \begin{proof}
 Notice that the bound $\frac{\norm{gv}}{\norm{v}} \leq \norm{g}$ is automatic.
 As for the other inequality, we will use Lemma \ref{lemma:bps}.

 Observe that the transformation $g^*$ is proximal with attracting eigenline $g_-^\perp$, repelling hyperplane $g_+^\perp$, and $\lambda_1 (g^*) = \lambda_1 (g)$.
 Moreover, $\mathcal{U}^+ (g^*, \varepsilon) \subset \mathcal{U}^- (g,\varepsilon)$ because $d (g_-^\perp, g_-) = \frac{\pi}{2}$.
 Particularly, $g^* g$ sends $\mathcal{U}^- (g,\varepsilon)$ into $\mathcal{U}^+ (g^*,\varepsilon)$, and has a fixed point in $\mathcal{U}^+ (g^*,\varepsilon)$.
 
 There must be a unique fixed point in $\mathcal{U}^+ (g^*,\varepsilon)$. Otherwise, $g^* g$ would fix another point corresponding to another singular vector, or it would fix a line. 
 The first case cannot occur, since another singular vector would lie $\frac{\pi}{2}$ away from the first one, contradicting the fact that both lie in an $\varepsilon$-ball. 
 Similarly, the second case contradicts that $g^*g (\mathcal{U}^- (g,\varepsilon)) \subset \mathcal{U}^+ (g^*,\varepsilon)$.
 This shows that $g^* g$ has a unique eigenvalue of maximal modulus in $\mathcal{U}^+ (g^*,\varepsilon)$.

 Notice that given $v \in g_+$, we get that:
        $$d (v, S(g)) = \frac{\pi}{2} - d(v,s(g)) \geq d(v, g_-) - \varepsilon > \varepsilon ,$$
 where we used the triangle inequality $d(v,s(g)) \leq \frac{\pi}{2} - d(v,g_-) + d (g_-^\perp, s(g))$, the fact that $s(g) \in B(g_-^\perp,\varepsilon)$ and $d(g_+,g_-) > 2\varepsilon$.
 Lemma \ref{lemma:bps} concludes the proof.
\end{proof}
\end{lemma}

On the other hand, we get the following:

\begin{lemma}
\label{lemma:agglomerated}
 Given $\varepsilon > 0$, and $g \in \End (V)$ a proximal element, we get that
 $$\sin (\varepsilon) \lambda_1 (g) \leq \frac{\norm{gv}}{\norm{v}},$$
 for any $v$ with $[v] \in \mathcal{U}^- (g,\varepsilon)$.
\end{lemma}

This will be a consequence of the following:

\begin{lem}
\label{lem:lawofsines}
 Let $p: V \to V$ a linear map such that $\dim p(V) =1$, and $\ker p$ is transverse to $p(V)$, then
    $$\frac{\norm{pv}}{\norm{v}} = \frac{\sin (\angle v, \ker p)}{\sin (\angle p(V), \ker p)}.$$
    \begin{proof}
 Given $v \in V$, up to restricting to $\mathbb{R} v \oplus \text{Im} p$, it is enough to assume that $\dim V = 2$.
 Notice that the origin, $p v$ and $v$ form a triangle such that the third side is parallel to $\ker p$.
 The law of sines produces the desired result.
    \end{proof}
\end{lem}

\begin{proof}[Proof of Lemma \ref{lemma:agglomerated}]
 Given $v \in \mathcal{U}^-(g,\varepsilon)$, we can split it as $v = v_+ + v_-$, where $v_+ \in g_+$ and $v_- \in g_-$.
 Notice that
 $$\norm{gv} = \norm{\lambda_1 (g) v_+ + g v_-} \geq \norm{ \lambda_1 (g) \pi v_+},$$
 where $\pi$ is the orthogonal projection onto $g_-^\perp$.
 Applying the previous lemma twice, we get that:
 
 $$\norm{\pi v_+} = \sin (\angle v_+, g_-) \norm{v_+} = \sin (\angle v, g_-) \norm{v},$$

 where we used that $v \to v_+$ is the image of a projection with image $g_+$ and kernel $g_-$.
 Putting everything together, we obtain the desired result.
\end{proof}

Combining Lemmas \ref{lemma:proximalestimate} and \ref{lemma:agglomerated}, we get the following corollary:

\begin{coro}
\label{lemma:busemannprox}
 Given $\varepsilon > 0$, and $g \in \End (V)$ $\varepsilon$-separated and proximal, we get that
 $$\sin^2 (\varepsilon) \norm{g} \leq \frac{\norm{gv}}{\norm{v}} \leq \norm{g},$$
 for any $v$ with $[v] \in \mathcal{U}^- (g,\varepsilon)$.
\end{coro}

This simple estimate is the heart of the proof of our main theorem
We close this section with a criterion to verify $\varepsilon$-separation.

\begin{lemma}
    \label{lem:implicaproximal}
 Given $g \in \End (V)$ a proximal transformation verifying $d(g_+,g_-) > 2\varepsilon$ then:
    $$\sin (\angle gv,g_+) \leq \frac{\norm{g|_{g_-}}}{\lambda_1 (g)} \frac{1}{\sin (\varepsilon)}.$$
 In particular, if $\frac{\norm{g|_{g_-}}}{\lambda_1 (g)} \leq \sin^2 (\varepsilon),$ then $g$ is $\varepsilon$-separated.
    \begin{proof}
 Notice that splitting $v \in V$ as $v = v_+ + v_-$, for $v_+ \in g_+$ and $v_- \in g_-$, we get that: 
        $$\norm{v_+} \sin (\angle v, g_+) = \sin (\angle v,g_-) \norm{v_-},$$
 by Lemma \ref{lem:lawofsines}. 
        
 In particular, $\frac{\norm{v_+}}{\norm{v_-}} \geq \sin \varepsilon$ if $[v] \in \mathcal{U}^- (g,\varepsilon)$.
 Applying $g$ to $v$ as above, we get that:

        $$\frac{\sin (\angle gv,g_-)}{\sin (\angle gv,g_+)} = \frac{\norm{(gv)_+}}{(gv)_-} \geq \frac{\lambda_1 (g)}{\norm{g|_{g_-}}} \frac{\norm{v_+}}{\norm{v_-}} \geq \frac{\lambda_1 (g)}{\norm{g|_{g_-}}} \sin \varepsilon.$$
 This implies the desired result.
    \end{proof}
\end{lemma}

\subsection{Symmetric spaces and Anosov representations}
\label{subsec:symmetric}

In this subsection, we will explain some generalities about Lie groups and symmetric spaces.
The proofs of all the stated results can be found in \cite[Chapter 6]{benoist2016random}.
For us, a Lie group $G$ will always be semisimple, real algebraic, and with non-compact factors. 
We denote by $\mathfrak{g}$ its Lie algebra.

\subsubsection{Cartan involutions and maximal compact subgroups}
The Killing form is a symmetric, bilinear, non-degenerate form given by $B (x,y) = \trace (ad (x) \circ ad (y)),$ for $x,y \in \mathfrak{g}$.
A \emph{Cartan involution} is an automorphism $o: \mathfrak{g} \to \mathfrak{g}$ with the property that $- B (o\cdot,\cdot)$ is positive definite.

Such an $o$ splits $\mathfrak{g} = \mathfrak{k}^o \oplus \mathfrak{p}^o$, where $\mathfrak{k}^o$ and $\mathfrak{p}^o$ are the plus and minus one eigenspaces of $o$, respectively.
In particular, $\mathfrak{k}^o$ is a subalgebra of $\mathfrak{g}$, and integrates to $K^o$ a connected Lie group.
Such a subgroup is a maximal connected compact subgroup of $G$, and one can show that any subgroup with these properties is conjugated to $K^o$ (in particular, its Lie algebra is the fixed point of a Cartan involution).

\subsubsection{Root space decomposition and Weyl chambers}

We define a \emph{Cartan subspace} as a maximal abelian subalgebra $\mathfrak{a} \subseteq \mathfrak{p}^o$.
The elements of $\mathfrak{p}^o$ are self-adjoint for $-B(o\cdot, \cdot)$, and thus diagonalizable. Particularly, $\mathfrak{a}$ acts simultaneously diagonally in $\mathfrak{g}$ via the adjoint representation. Therefore, $\mathfrak{g}$ decomposes as
$$\mathfrak{g} = \mathfrak{g}_0 \oplus \bigoplus_{\alpha \in \Sigma} \mathfrak{g}_\alpha,$$
where, for $\alpha \in \mathfrak{a}^*$, $\mathfrak{g}_\alpha = \{ x \in \mathfrak{g} : [h,x] = \alpha(h) x\: \forall h \in \mathfrak{a}\}.$
The set of \emph{restricted roots} $\Sigma$ is defined as $$\Sigma = \{\alpha \in \mathfrak{a}^* \setminus \{0\} : \mathfrak{g}_\alpha \neq \{0\}\}.$$
The spaces $\mathfrak{g}_\alpha$ are called the \emph{root spaces}.

The closures of connected components of
$\mathfrak{a} \setminus \bigcup_{\alpha \in \Sigma} \ker \alpha$ are called \emph{Weyl chambers}. 
We choose from now on a Weyl chamber and denote it by $\mathfrak{a}^+$. 
Moreover, the interior of $\mathfrak{a}^+$ will be denoted $\mathfrak{a}^{++}$.
Such a choice of $\mathfrak{a}^+$ induces an order in $\Sigma$ by defining the \emph{set of simple roots} $\Pi = \{ \alpha \in \Sigma : \ker \alpha \cap \mathfrak{a}^+ \neq \{0\}\}$. Positive linear combinations of simple roots are called \emph{positive roots}.
The \emph{Weyl group} $W$ of $G$ is defined as the group generated by the orthogonal reflections in the subspaces $\{\ker \alpha: \alpha \in \Sigma\}$. 

There is a single longest element in the Weyl group denoted by $w_0$. The opposition involution is defined via
$$\iota: \mathfrak{a} \to \mathfrak{a}, \qquad \iota(v) = -w_0(v).$$

It is the only involution preserving $\mathfrak{a}^+$ and sending a simple root to the opposite. In $SL_n(\RR)$, it coincides with $\iota(a_1, \ldots, a_n) = (-a_n, \ldots, -a_1)$.

\subsubsection{Parabolic subalgebras}

Given $\Theta \subset \Pi$, we will define a pair of \emph{parabolic $\Theta$-opposite subalgebras} by
$$
\begin{aligned}
\mathfrak{p}_\Theta = \mathfrak{g}_0 \oplus \bigoplus_{\alpha \in \Sigma^+} \mathfrak{g}_\alpha \oplus \bigoplus_{\alpha \in \langle \Pi - \Theta \rangle} \mathfrak{g}_{-\alpha}, & \: \check{\mathfrak{p}}_\Theta = \mathfrak{g}_0 \oplus \bigoplus_{\alpha \in \Sigma^+} \mathfrak{g}_{-\alpha} \oplus \bigoplus_{\alpha \in \langle \Pi - \Theta \rangle} \mathfrak{g}_{\alpha}
\end{aligned}
$$

We denote by $P_\Theta$ and $\check{P}_\Theta$ their respective Lie groups, which we call \emph{opposite parabolic subgroups}.

The quotient space $\mathcal{F}_\Theta = G/P_\Theta$ is called the \emph{flag manifold} associated to $P_\Theta$, and we call $\check{\mathcal{F}_\Theta}$ the opposite flag manifold. One can show that the action of $K$ in $\mathcal{F}_\Theta$ is transitive.   

\subsubsection{The symmetric space of $G$}
\label{subsubsec:symmetric}

All the objects above have geometric interpretations in terms of the symmetric space of $G$.
We will define the \emph{symmetric space} $\mathbb{X}$ of $G$ as the set of Cartan involutions of $\mathfrak{g}$. 
Fixing a basepoint $o \in \mathbb{X}$, one can identify $\mathbb{X}$ with the coset space $G/K$, for $K = K^o$.
The Killing form on $\mathfrak{g}$ induces a non-positively curved Riemannian metric on $\mathbb{X}$.

A choice of a Cartan subspace $\mathfrak{a} \subset \mathfrak{p}^o$ integrating to $A$ defines a submanifold $A K \subset G/K$.
These submanifolds are totally geodesic, have zero curvature, and they are maximal with respect to these properties.
We will call them \emph{maximal flats} (or just flats).
Notice that flats are identified with $\mathfrak{a}$ via the exponential map, and therefore, Weyl chambers get identified with a region of this submanifold.

One can prove that for any $x \in \mathbb{X}$, its $K$ orbit intersects a chosen Weyl chamber in a unique point.
This is equivalent to the fact that any $g \in G$ belongs to:
$$g \in K \exp (\kappa (g)) K,$$
for a uniquely defined $\kappa (g)\in \mathfrak{a}^+$.
We will call $\kappa: G \to \mathfrak{a}^+$ the \emph{Cartan projection}.
Notice that $\norm {\kappa (g)} = d (o,go)$, for the Riemannian norm.

There is a ``stabilized'' version of the Cartan projection called the \emph{Jordan} projection $\lambda :G \to \mathfrak{a}^+$ defined via
$$\lambda (g) = \lim_n \frac{1}{n} \kappa (g^n).$$
This can also be characterized as the unique element of $\mathfrak{a}^+$ such that the hyperbolic piece of $g$ in the Jordan decomposition is conjugated to $\exp(\lambda (g))$.

Any ray based at $o \in \mathbb{X}$ and converging to $\xi \in \partial \mathbb{X}$ lies in the $K$-orbit of an element of our model Weyl chamber defined by $\mathfrak{a}^+$.
Rays in our model Weyl chamber are in correspondence with points of the projectivization $\mathbb{P} (\mathfrak{a}^+)$ of the cone $\mathfrak{a}^+$.
We will decompose $\mathbb{P} (\mathfrak{a}^+)$ along $\Theta-$\emph{facets} defined as the projectivization of the intersection of:
$$\mathfrak{a}_\Theta = \bigcap_{\alpha \in \Pi \setminus \Theta} \ker \alpha,$$
with $\mathfrak{a}^+$, where $\Theta \subset \Pi$ is a subset of simple roots.
An \emph{open facet} will be the set of points in the $\Theta-$facet, whose rays are disjoint from $\mathfrak{a}_{\Theta'}$ for $\Theta' \subset \Theta$.
One can show that the stabilizer of a ray in an open facet is precisely $P_\Theta$, the parabolic group introduced earlier.

We will denote by $p_\Theta : \mathfrak{a} \to \mathfrak{a}_\Theta$ the orthogonal projection.
We can use it to define restricted Cartan and Jordan projections $\kappa_\Theta = p_\Theta \circ \kappa$, and $\lambda_\Theta = p_\Theta \circ \lambda$.

Given a point in $\partial \mathbb{X}$ lying in the open facet defined by $\Theta = \Pi$, and $g \in G$, one can show that the horosphere based at $\xi$ passing through $g K$ intersects the Weyl chamber in a unique point.
From a Lie group perspective, this is equivalent to the fact that $G$ is diffeomorphic to $K \times \exp \mathfrak{a} \times U$, for $U$ the unipotent radical of the parabolic $P = P_\Pi$.
In particular, given $\xi \in G/P$, we know that $\xi = k P$ for some $k \in K$, we can then define $\sigma : G \times G/P \to \mathfrak{a}$ via
$$g k \in K \exp \sigma (g,k) U.$$
This function is called the \emph{Iwasawa cocycle}, and was defined by Quint in \cite{quint2002patterson} (see also \cite{benoist2016random}).
The nomenclature follows from the fact that
$$\sigma (g g',\xi) = \sigma (g, g'\xi) + \sigma (g',\xi).$$
Given any $x \in \mathfrak{a}^{++}$, one can prove that
$$\sigma (g,\xi) = \lim_{t\to \infty } \kappa (gk e^{tx}) - tx,$$
where $\xi = k P$ as before.
This relates the Iwasawa cocycle with the Busemann functions on the symmetric space.

Given $\Theta \subset \Pi$, we will refer to the \emph{partial Iwasawa cocycle} $\sigma_\Theta : G \times G/P_\Theta \to \mathfrak{a}$ as $\sigma_\Theta = p_\Theta \circ \sigma$.
It is easy to show that this is also a cocycle.

\subsubsection{Fundamental representations}

Many computations in the symmetric space are reduced to elementary linear algebra by using fundamental representations.
Recall that for a given simple root $\alpha \in \Pi$, there exists a unique $\omega_\alpha \in \mathfrak{a}^*$ such that $\omega_\alpha (h_\beta) = \delta_{\alpha \beta}$, for $\beta \in \Pi$ and $h_\beta$ a normalized co-root with respect to the inner product induced by the Killing form.
We say that $\omega_\alpha$ is the \emph{fundamental weight} associated to $\alpha$.
These form a basis for $\mathfrak{a}^*$. 
Moreover, each fundamental weight appears as the highest weight vector of an irreducible representation (up to covers).

\begin{thm}[Tits]
 For every $\alpha \in \Pi$, there exists an irreducible representation $\rho_\alpha : G \to GL (V^\alpha)$ with highest weight vector $\overline{\omega_\alpha}$ being an integer multiple of $\omega_\alpha$.
\end{thm}

Given $o \in \mathbb{X}$, each vector space $V^\alpha$ can be equipped with a $K^o$-invariant inner product which we call \emph{adapted}.
The following lemma will be key for the upcoming computations:

\begin{lemma}[Lemma 5.32 of \cite{benoist2016random}]
\label{lemma:fundamentalrep}
 Given $\rho_\alpha$ the representation associated to the fundamental weight $\omega_\alpha$, and $||\cdot ||$ an adapted inner product, then
    \begin{enumerate}
        \item $\overline{\omega_\alpha} (\kappa (g)) = \log \norm{\rho_\alpha (g)}$.
        \item $\overline{\omega_\alpha} (\lambda (g)) = \log (\lambda_1 (\rho_\alpha (g)))$.
        \item $\overline{\omega_\alpha} (\sigma (g, \xi)) = \log \frac{\norm{ \rho_\alpha (g) v}}{\norm{v}},$ where $v \in V^\alpha$ belongs to $h L^\alpha$, with $\xi = hP$ ($P$ being the standard minimal parabolic), and $L^\alpha$ the highest weight space. 
    \end{enumerate}
\end{lemma}

Given $\Theta \subset \Pi$, one can prove that for each $\alpha \in \Theta$, the parabolic $P_\Theta$ fixes the highest weight line $L^\alpha \subseteq V^\alpha$. 
This defines a $G$-equivariant map
$$\iota^\alpha: \mathcal{F}_\Theta \to \PP(V^\alpha)$$
by $\iota^\alpha(gP_\Theta) = \rho_\alpha(g)L^\alpha$.
Analogously, one can take $\rho_\alpha^*: G \to GL((V^\alpha)^*)$, and define $\iota^{\alpha}_-: \mathcal{F}^{op}_\Theta \to \PP((V^\alpha)^*)$.  
One can check that the map $\iota_\Theta =  \prod_{\alpha \in  \Theta} \iota^\alpha$ is an embedding from $\mathcal{F}_\Theta$ to $\prod_{\alpha \in \Theta}\PP(V^\alpha)$. 
It is called the \emph{Plücker embedding}. 

We say that a pair of flags $(x, y) \in \mathcal{F}_\Theta \times \mathcal{F}_\Theta^{op}$ is \emph{transverse} if for all $\alpha \in \Theta$, $\iota^\alpha(x)$ does not belong to $\iota_-^\alpha(y)$ (thought of as a hyperplane in projective space).
More geometrically, one can check that this is equivalent to the existence of a geodesic in the symmetric space converging to a point at infinity stabilized by each of the parabolics at $\infty$ and $-\infty$, respectively.

\subsection{$\Theta$-Loxodromic elements}

Fix $o \in \mathbb{X}$ and a Weyl chamber $\mathfrak{a}^+ \subset T_o \mathbb{X}$.
Given $\alpha \in \Pi$ a simple root, we get $\rho_\alpha : G \to PGL (V^\alpha)$ the associated fundamental representation as in the previous section.
The choice of $o \in \mathbb{X}$ induces an \emph{adapted inner product} on $V^\alpha$.
We will fix a subset of simple roots $\Theta \subset \Pi$ for the remainder of the section.

\begin{defi}
\label{defi:separatedloxodromic}
 Given $\Theta \subset \Pi$, we say that $g \in G$ \emph{$\Theta$-loxodromic} if $\rho_\alpha (g)$ is proximal for every $\alpha \in \Theta$. 
 Moreover, if the $\rho_\alpha(g)$ are all $\varepsilon$-separated, we say that $g$ is \emph{$(\Theta,\varepsilon)$-separated}.
\end{defi}

We can endow the partial flag manifold $G/P_\Theta$ with a metric by pulling back the induced metric on $V^\alpha$ for the good inner products via the Plücker embedding $$\iota_\Theta: \mathcal{F}_\Theta \to \prod_{\alpha \in \Theta} \mathbb{P} (V^\alpha).$$
This distance coincides with the visual metric induced by the basepoint $o \in \mathbb{X}$.

Any $\varepsilon$-separated element $g$ fixes a pair of flags $(g_+,g_-) \in \mathcal{F}_\Theta \times \mathcal{F}_\Theta^{op}$ called the attractive (respectively, repelling) flags.
Under the Plücker embedding, the attractive flag $g_+$ corresponds to the tuple $(\rho_\alpha (g)_+)_{\alpha \in \Theta}$, and the repelling flag $g_-$ to the tuple of hyperplanes $(\rho_\alpha^* (g^{-1})_+)_{\alpha \in \Theta}$.
Moreover, one can show that any $\Theta$-loxodromic verifies that $\lambda_\Theta (g) \in \mathfrak{a}_\Theta$ lies in the interior of the positive cone bounded by $\ker \alpha$ for $\alpha \in \Theta$.
We will denote this cone as $\mathfrak{a}_\Theta^+$.
One can use this cone to define a partial order in $\mathfrak{a}_\Theta$, via
$$x \leq_\Theta y \Leftrightarrow y - x \in \mathfrak{a}_\Theta^+.$$
If we ever forget the subscript, it means that we are taking $\Theta = \Pi$.
Just like for $\varepsilon$-separated proximal elements, we will denote
$$\mathcal{U}^-_\Theta (g,\varepsilon) = \{ f \in \mathcal{F}_\Theta: d(\iota^\alpha (f), \rho_\alpha (g)_-) \geq \varepsilon, \: \forall \alpha \in \Theta\}.$$
This is the complement of an $\varepsilon$-neighborhood of the flags in $\mathcal{F}_\Theta$ that are non-transverse with respect to $g_-\in \mathcal{F}_\Theta^{op}$, also called the Schubert locus.
The estimates of Corollary \ref{lemma:busemannprox} and Lemma \ref{lemma:agglomerated}, along with the discussion on the fundamental representations of Lemma \ref{lemma:fundamentalrep} lead us to the following proposition

\begin{prop}
\label{coro:estimatesgeneralcase}
 Let $\varepsilon >0$, and $g \in G$ be $\Theta$-loxodromic, then, given $R$ the unique vector such that $\overline{\omega}_\alpha (R) =1$ for all $\alpha \in \Pi$, and $R_\varepsilon = - \log \sin (\varepsilon) R$. Then
 \begin{itemize}
    \item $\lambda_\Theta (g) - \sigma_\Theta (g,f) \leq_\Theta R_{\varepsilon}$, for every $f \in \mathcal{U}^-_\Theta (g,\varepsilon)$.
    \item If $g$ is $\varepsilon$-separated, then 
    $$0 \leq_\Theta \kappa_\Theta (g) - \sigma_\Theta (g,f) \leq_\Theta 2 R_\varepsilon,$$ 
 for all $f \in \mathcal{U}^-_\Theta (g,\varepsilon)$.
 \end{itemize}
\end{prop}

This is the key estimate that we will need for our main theorem.
This proposition should be compared with \cite[Lemma 3.4]{quint2003indicateur}.

%% file: sections/thermodynamic.tex
\section{Thermodynamic Formalism and applications}
\label{sec:thermodynamic}

The goal of this section is to review the notion of Markov coding of a Gromov hyperbolic group and adapt the multicone criterion developed in \cite{bochi2019anosov} to this setup.
This will allow us to import techniques from symbolic dynamics to measure quantitative properties of Anosov representations.
In particular, following Quint \cite{quint2005schottky}, we define a Busemann potential and show how to compute critical exponents from it.

\subsection{Markov structures and sofic shifts}
\label{sec:markov}

In this section, we will review the notion of a strong Markov structure adapted to a Gromov hyperbolic group $\Gamma$ equipped with a symmetric set of generators $S$, inducing a word-length function $|\cdot|_S$.
Such structure will allow us to codify the group action on itself, and on its boundary, allowing us to use techniques from thermodynamic formalism. This notion was introduced in \cite{ghys1990sur}, and we follow the terminology used in Cantrell \cite{cantrell2025mixing}, and Cantrell-Tanaka \cite{cantrell2025manhattan}.

Recall that a \emph{directed graph} is a tuple $(\mathcal{V},\mathcal{E}, i,f)$, where $\mathcal{V}$ is called the set of vertices, $\mathcal{E}$ the set of edges, and $i,f: \mathcal{E} \rightarrow \mathcal{V}$ functions denoting the initial, and final vertices corresponding to a directed edge $e \in \mathcal{E}$.
We will denote $v \xrightarrow{e} v'$ if $i(e) =v$ and $f(e) = v'$.
We will assume that directed graphs are finite.

Given $S$ a finite set, we say that a \emph{graph labeled on $S$} is a directed graph equipped with a function $\pi: \mathcal{E} \rightarrow S$ called a labeling.
We will be interested in the case where $S$ is a symmetric system of generators of a Gromov hyperbolic group $\Gamma$.
In this case, given a graph $\mathcal{G}$ labeled on $S$, and a path
$$v_0 \xrightarrow{\pi(e_1)} v_1 \xrightarrow{\pi(e_2)} \ldots \xrightarrow{\pi(e_{n})} v_n,$$ 
we define the evaluation of the path $ev (e_1,\ldots,e_n)$ as the product $\pi (e_n) \ldots \pi (e_1)$.

\begin{defi}
 Given $(\Gamma, S)$ a Gromov hyperbolic group with a symmetric set of generators, we will call a labeled graph $\mathcal{G}$ equipped with a special vertex $*$ a \emph{strong Markov structure} if it satisfies the following properties:
   \begin{enumerate}
      \item For each vertex $v \in \mathcal{V}$, there exists a path starting at $*$ and ending at $v$.
      \item For each path $(e_1,\ldots,e_n)$ with $i(e_1) = *$, we get that the length $|ev(e_1,\ldots,e_n)|_S$ is $n$.
      \item Every $\gamma \in \Gamma$ is equal to $ev (e_1,\ldots,e_n)$ for a unique path $(e_1,\ldots,e_n)$ starting at $*$.
   \end{enumerate}
\end{defi}

One can prove that such a strong Markov structures always exist, see \cite[Section 3.2]{calegari2011ergodic} for a proof (such result is usually attributed to Cannon \cite{cannon1984combinatorial}).

Notice that each labeled directed graph $\mathcal{G}$ has a unique labeled subgraph $\mathcal{G}^r$ that is \emph{maximally recurrent} in the sense that it is a union of cycles, and it is not contained in any other subgraph with this property.
The finiteness of the graph $\mathcal{G}$ implies the following:

\begin{lemma}
\label{lemma:densityandfiniteness}
 Given $\mathcal{G}$ a strong Markov structure for $(\Gamma,S)$, then there exists $c >0$ such that the evaluation of paths $(e_1,\ldots,e_n)$ in the recurrent graph $\mathcal{G}^r$ is $c-$dense in $\Gamma$.
 That means, for every $\gamma \in \Gamma$, there is $\gamma' = ev(e_1,\ldots,e_n)$ for which $d(\gamma,\gamma') \leq c$.
 Moreover, each element of $\Gamma$ is the evaluation of finitely many finite paths in $\mathcal{G}^r$, and this number is independent of the chosen element.
\end{lemma}

We will denote the set of conjugacy classes of $\Gamma$ as $[\Gamma]$.
The following result due to Cantrell tells us that each conjugacy class has nice representatives in the recurrent part of a strong Markov structure:

\begin{prop}[{\cite[Corollary 3.5]{cantrell2025mixing}}]
\label{prop:cantrell}
 Let $\mathcal{G}$ a strong Markov structure for $(\Gamma,S)$, and $\mathcal{C}$ a connected component of $\mathcal{G}^r$.
 Then, there exists $N>0$ such that for any non-torsion conjugacy class $[g] \in [\Gamma]$ there is a periodic path $(e_0,\ldots, e_l)$ (i.e., with $i(e_0) = f(e_l)$) contained in $\mathcal{C}$ for which $ev (e_0,\ldots,e_l)$ belongs to one of $[g^{\pm N}]$.   
\end{prop}

A strong Markov structure for $(\Gamma, S)$ can be used to define a dynamical system via
$$X_\mathcal{G} := \{ (s_i) \in S^\mathbb{Z} : \exists (e_n)_{n \in \mathbb{Z}} \text{ biinfinite path s.t. }\pi (e_n) = s_n,\: \forall n \in \mathbb{Z}\},$$
and equipping it with the \emph{shift map} $T: X_\mathcal{G} \rightarrow X_\mathcal{G}$ sending $(x_n)_{n \in \mathbb{Z}}\to (x_{n+1})_{n\in \mathbb{Z}}$.
Since bi-infinite paths are contained in $\mathcal{G}^r$, such a definition depends solely on the recurrent part of the graph.
We can analogously define a unilateral shift $X_\mathcal{G}^+$ by taking the positive labels of a bi-infinite path.

Notice that we can define a map $p : X_\mathcal{G}^+ \rightarrow \partial \Gamma$ by sending
$$\underline{x} \to \lim_n x_0^{-1} \ldots x_n^{-1},$$
which is well defined because $\gamma_n = x_0^{-1}\ldots x_n^{-1}$ is a geodesic on $\Gamma$.

Such dynamical systems are called \emph{sofic shifts} (see \cite{lind2021introduction} for more information).
These systems are always finite factors of subshifts of finite type.
To see this, notice that one can define $\Sigma_\mathcal{G}$ as the space of infinite sequences of edges $(e_n)_{n \in \mathbb{Z}}$ corresponding to biinfinite paths.
Since each vertex $v \in \mathcal{V}$ has at most one edge for each label (by the third condition of a strong Markov structure), we get that the projection map $\Sigma_{\mathcal{G}} \rightarrow X_\mathcal{G}$ is finite to one (with fibers being bounded by the number of vertices).

\begin{ex}
\label{ex:codings}
 Suppose $\Gamma = F_2$, Figure \ref{fig:markov_partition} represents two labeled graphs defining strong Markov structures for the system of generators $S = \{a,b,a^{-1},b^{-1}\}$, for $a,b$ a free basis, and $S = \{a,b,c, a^{-1},b^{-1},c^{-1}\}$ with $abc = id$.
\begin{figure}[h!]
   \begin{center}
      \includegraphics[scale=1.2]{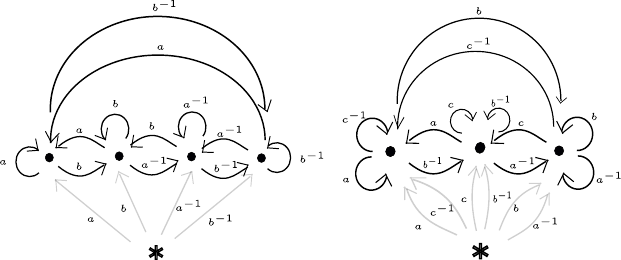}
   \end{center}
   \caption{Strong Markov structures for the free group. Arrows in grey are not in the recurrent part of the graph.}
   \label{fig:markov_partition}
\end{figure}

Both of these graphs have the property that edges with the same label end in the same vertex.
This translates to the fact that the sofic shift $X_\mathcal{G}$ is in fact, a subshift of finite type.
\end{ex}

\subsection{Anosov representations and the multicone criterion}

In this section, we briefly recall the definition of Anosov representations. 
We will then explain a characterization via invariant multicone families adapted to strong Markov partitions, based on the work of Bochi-Potrie-Sambarino \cite{bochi2019anosov}.

Following \cite{bochi2019anosov} and \cite{kapovich2017anosov}, we can think of the Anosov conditions as a strengthening of a group being quasi-isometrically embedded in the symmetric space.

\begin{defi}
 Let $G$ be a semisimple Lie group with no compact factors, $\mathfrak{a}$ a Cartan subspace, and $\Theta \subset \Pi$ a subset of simple roots.
 We say that a representation $\rho : \Gamma \rightarrow G$ is $\Theta$-\emph{Anosov} if there exist $A,B >0$ such that:
    $$\alpha (\kappa (\rho (g))) \geq A |g| - B,\: \forall g \in \Gamma, \alpha \in \Theta,$$
 where $|\cdot|$ is a choice of word metric on $\Gamma$, and $\kappa$ is the Cartan projection.    
\end{defi}

In \cite{kapovich2017anosov}, the authors show that if a group $\Gamma$ admits an Anosov representation, then it is Gromov hyperbolic.
Moreover, the limit set on $\mathcal{F}_\Theta$ of a $\Theta-$Anosov representation can be equivariantly parametrized by $\partial\Gamma$.

More precisely, there exists a pair of maps $(\xi,\xi^{op}): \partial \Gamma \rightarrow \mathcal{F}_\Theta \times \mathcal{F}_\Theta^{op}$, with $\xi(x)$ transverse to $\xi^{op}(y)$ for $x \neq y$. These maps are dynamics-preserving, particularly implying that for $\gamma \in \Gamma$ we have that $\xi (\gamma^+)$ and $\xi^{op} (\gamma^-)$ are the attractors and repellors for the $\Theta-$loxodromic element $\rho (\gamma)$.
We will refer to $(\xi,\xi^{op})$ as the \emph{limit maps}.
See \cite{bochi2019anosov} or \cite{kapovich2017anosov} for a proof of this characterization.

We will now see how the Anosov property can be certified using a strong Markov structure for $(\Gamma, S)$ and an invariant family of multicones, as described below.

\begin{defi}
 Let $\mathcal{G}$ be a strong Markov structure for $(\Gamma,S)$, $\rho : \Gamma \rightarrow G$ a representation and $\Theta \subset \Pi$.
   \begin{itemize}
      \item A \emph{multicone} in $\mathcal{F}_\Theta$ is an open subset $M$ with the property that, for every $\alpha \in \Theta$, $\iota^\alpha (M)$ is an open subset of $\mathbb{P} (V_\alpha)$ contained in the complement of some hyperplane.
      \item An \emph{invariant family of  $\Theta-$multicones} for $\rho$ adapted to $\mathcal{G}$ is a collection $\{M_v\}_{v \in \mathcal{V}}$ of multicones indexed in vertices of $\mathcal{G}$ with the property that whenever we get an edge $v_0 \xrightarrow{e} v_1$, we get:
      $$\overline{\rho(\pi (e)^{-1}) M_{v_1}} \subset M_{v_0}.$$ 
   \end{itemize}
\end{defi}

The following criterion follows from the techniques used in \cite{bochi2019anosov}.
Since our setup is slightly different from theirs, we include a roadmap of the proof for the reader's convenience.

\begin{thm}
   \label{thm:coding}
 Let $(\Gamma, S)$ be a hyperbolic group with a symmetric set of generators.
 A representation $\rho$ is $\Theta-$Anosov if and only if there is an invariant family of $\Theta-$multicones adapted to $\mathcal{G}$.
 Moreover, given $p : X_\mathcal{G}^+ \rightarrow \partial \Gamma$ the coding of the boundary, we get:
   $$\xi (p(\underline{x})) = \lim_n \rho (x_0^{-1}\ldots x_n^{-1}) P_{n+1},$$
 for $\xi$ the limit map of $\rho$ and $P_{n+1} \in M_{v_{n+1}}$ in the vertex $f(e_n)$, for $e_n$ defining the labeling $x_n$.
 Particularly, $\xi (p (\underline{x})) \in M_{v_0}$ the initial vertex of the path defining $\underline{x}$.
\begin{proof}
 Up to composing with a fundamental representation, we can assume that $G = SL_n (\mathbb{R})$, and study the projective Anosov case (i.e., $\Theta = \{\alpha_1\}$ such that $\mathcal{F}_\Theta$ is projective space).
   
 Let $\mathcal{G}^{op}$ be the labeled graph obtained by reversing the orientation of every edge, and changing the label $s$ to $s^{-1}$.
 Such a labeled graph induces a (bi-infinite) sofic shift $X_{\mathcal{G}^{op}}$.
 An invariant family of multicones for $\rho$ adapted to $\mathcal{G}$ is the same as having an invariant family of cones for the linear cocycle $T_\rho: X_{\mathcal{G}^{op}} \times \mathbb{R}^n \rightarrow X_{\mathcal{G}^{op}} \times \mathbb{R}^n$ sending:
$$(\underline{x}, v)\to (T\underline{x}, \rho(x_0) v).$$
 A standard argument (see for instance \cite[Theorem 2.6]{crovisier2015introduction}) implies that such cones exist iff there exists a pair $(E^{cu}, E^{cs}): X_{\mathcal{G}^{op}} \rightarrow (\mathbb{P}(\mathbb{R}^n), \mathbb{P}((\mathbb{R}^n)^*))$ defining a dominated splitting for $T_\rho$.

 By the work of Bochi-Gourmelon \cite{bochi2009some}, the existence of dominated splittings is in turn equivalent to having exponential gaps between the first two singular values for every finite product $\rho (ev (e_1,\ldots,e_l)^{-1})$ corresponding to a path $ev (e_1,\ldots,e_l)$ in the recurrent piece $\mathcal{G}^r$.
 Since these elements are $c-$dense in $\Gamma$ by Lemma \ref{lemma:densityandfiniteness}, this is equivalent to having singular value gaps in the whole group, which is precisely the Anosov property.

 Finally, Lemma 4.7 and Proposition 5.4 of \cite{bochi2019anosov} imply that if $x$ is the limit of a geodesic ray $\lim_n x_0^{-1}\ldots x_l^{-1}$, then the limit map is defined by
 $$\{\xi (x) \} = \bigcap \rho(x_0^{-1}\ldots x_n^{-1})M_{v_n},$$
 which implies the last claim.
\end{proof}
\end{thm}

\subsection{Critical exponents via sofic shifts}

The goal of this section is to explain how we can use thermodynamic formalism in the sofic shift to compute critical exponents for an Anosov representation into a Lie group $G$. 
As before, we assume a choice of $\mathfrak{a}$, a Cartan subspace with a fixed Weyl chamber $\mathfrak{a}^+$. 

The limit cone $\mathcal{L}_\Gamma$ of a subgroup $\Gamma$ is the minimal closed cone containing the Jordan projection of every element in $\Gamma$.
It can analogously be defined as the asymptotic cone generated by the Cartan projection of elements of $\Gamma$, i.e.,
$x \in \mathcal{L}_\Gamma$ if and only if there exists $t_n \to 0$, and $\gamma_n \in \Gamma$ going to infinity such that $x = \lim_n t_n \kappa (\gamma_n)$ (see \cite[Remark 6.3]{benoist2016random}).

We say that $\varphi \in \mathfrak{a}^*$ is \emph{positive in the limit cone} if $\varphi|_{\mathcal{L}_\Gamma - \{0\}} > 0$.

\begin{defi}
 Given $\Gamma \subset G$ a discrete subgroup of $G$, and $\varphi \in \mathfrak{a}^*$, positive in the limit cone of $\Gamma$, we define the \emph{critical exponent} associated to $\varphi$ as
 $$
 h_\Gamma (\varphi)
 =\limsup_{L\to\infty}\frac{1}{L}\log\left(\#\{\gamma\in\Gamma:\ \varphi(\kappa(\gamma))\le L\}\right).
    $$
\end{defi}
This definition differs from the one given in the introduction, where the critical exponent is defined using the Jordan projection and conjugacy classes. However, the two notions coincide (see, for instance, \cite{glorieux2023hausdorff}).

Note that the number $h_\Gamma (\varphi)$ coincides with the critical exponent of the Dirichlet series given by
\[
s \to \sum_{\gamma \in \Gamma} e^{-s\varphi(\kappa(\gamma))}.
\]
We will be exclusively concerned with the case when $\Gamma$ is the image of an Anosov representation. 

Fix $\mathcal{G}$ a strong Markov system for $(\Gamma,S)$ inducing a (unilateral) sofic shift $X_{\mathcal{G}}^+$.
Given a finite sequence $w = (x_0,\ldots,x_{n-1}) \in S^{n}$, we will define its \emph{cylinder} $[w] \subset X_\mathcal{G}^+$ to be the set of all sequences whose first $n$ entries coincide with $w$ (perhaps empty if such word is not admissible).

Given $f \in C(X_\mathcal{G}^+)$ a continuous function, we will denote by $\Sigma_n f = \sum_{i=0}^{n-1} f \circ T^i$ its partial Birkhoff sum.
Recall the following notion from thermodynamic formalism (see for instance \cite{viana2016foundations})

\begin{defi}
   \label{def:pressure}
 Given $f \in C(X_\mathcal{G}^+)$ a continuous function, we say that the \emph{pressure of $f$} is the number:
   $$P(f) = \lim_{n \to \infty}\frac{1}{n} \log \left( \sum_{w \in S^n} \exp (\sup_{\underline{x} \in [w]}\Sigma_n f(\underline{x})) \right),$$
 taking the sup as $-\infty$ if the set $[w]$ is empty.
\end{defi}

We will refer to functions on $X_\mathcal{G}^+$ as \emph{potentials}.
Notice that pressure is monotonic in the sense that given $f \leq g$ potentials, then $P(f) \leq P(g)$.

Following Quint \cite{quint2003indicateur}, we will now define a $\mathfrak{a}_\Theta^+$-valued potential which, when composed with $\varphi \in (\mathfrak{a}_\Theta)^*$, defines potentials whose pressure recovers the critical exponents, as we will soon see.

For $\sigma, \kappa$ and $\lambda$ the Iwasawa, Cartan and Jordan projections, we denote $\sigma_\Theta = p_\Theta \circ \sigma, \kappa_\Theta = p_\Theta \circ \kappa$, and $\lambda_\Theta = p_\Theta \circ \lambda$. 

\begin{defi}
\label{defi:busemannpotential}
 Given $\mathcal{G}$ a strong Markov structure for $(\Gamma,S)$ and $\rho : \Gamma \rightarrow G$ a $\Theta-$Anosov representation, we define the \emph{Busemann potential} $B: X_\mathcal{G} \rightarrow \mathfrak{a}_\Theta^+$ as
 $$B (\underline{x}) = \sigma_\Theta(\rho (x_0)^{-1}, \xi \circ p (T (\underline{x}))),$$
 where $\sigma_\Theta$ is the $\Theta-$Iwasawa cocycle, $p: X_\mathcal{G}^+ \rightarrow \partial \Gamma$ is the coding of the limit set as described in Section \ref{sec:markov}, and $\xi :\partial \Gamma \rightarrow \mathcal{F}_\Theta$ is the limit map.
\end{defi}

It is not difficult to see that $p: X_\mathcal{G}^+ \rightarrow \partial \Gamma$ is a H\"older map.
Moreover, the limit map $\xi : \partial \Gamma \rightarrow \mathcal{F}_\Theta$ is also H\"older by \cite[Theorem A.15]{bochi2019anosov}.
These two facts and the analyticity of $\sigma_\Theta$ imply that $B$ is an H\"older potential. Given $\varphi \in \mathfrak{a}_\Theta^*$, we denote $B_\varphi = \varphi \circ B$.

The following result links the critical exponent of the Poincar\'{e} series with the pressure:

\begin{prop}
   \label{prop:poincareradioespectral}
 Let $\rho : \Gamma \rightarrow G$ a $\Theta-$Anosov representation, $\mathcal{G}$ a strong Markov coding for $(\Gamma,S)$ and $B$ the Busemann potential. 
 Given $\varphi \in \mathfrak{a}_\Theta^*$ positive in the limit cone of $\rho$, we get that the equation $P(-s B_\varphi)= 0$ has a unique solution at $s = h_\rho (\varphi)$. 
\end{prop}

The key ingredient is the following lemma: 

\begin{lemma}
   \label{lemma:multiconesbusemann}
 Let $\rho: \Gamma \to G$ be a $\Theta$-Anosov representation, and $\mathcal{G}$ a strong Markov coding for $(\Gamma,S)$.
 Then, there exists $C>0$ and $N \in \mathbb{Z}_{>0}$ for which: 
   $$\|\sigma_\Theta(\rho(x_{0}^{-1}\ldots x_{n-1}^{-1}), \xi \circ p (T^n(\underline{x}))) - \kappa_\Theta(\rho(x_0^{-1}\ldots x_{n-1}^{-1}))\| \leq C,$$
 for every $\underline{x} \in X_\mathcal{G}^+$ and $n > N$. 
\begin{proof}
 Composing $\rho$ with a fundamental representation, we obtain $\rho_\alpha \circ \rho : \Gamma \rightarrow SL (V_\alpha)$ for every $\alpha \in \Theta$, which we abuse notation and denote simply by $\rho$.
 Given $g \in SL(V_\alpha)$, we denote by $U(g)$ the space spanned by the first singular vector, and by $S(g)$ its orthogonal complement.

 By Lemma \ref{lemma:fundamentalrep}, the bound we want to prove is equivalent to showing the bound 
   $$|\log ||\rho (x_0^{-1} \ldots x_{n-1}^{-1}) (v)|| - \log \sigma_1 (\rho (x_0^{-1} \ldots x_{n-1}^{-1}))| \leq C,$$
 for $v \in \xi (T^n (\underline{x})) = \lim_{l\to \infty} U(x_n^{-1}\ldots x_l^{-1})$ (by Lemma 4.7 of \cite{bochi2019anosov}) of unit norm with respect to $|| \cdot ||$ any adapted norm in $V_\alpha$.

 Lemma 2.5 of \cite{bochi2019anosov} shows that there exists $\delta >0$, and $N \in \mathbb{Z}_{\geq 0}$ for which the angle

   $$\angle \left( U(\rho (x_n^{-1}\ldots x_l^{-1})), S(\rho (x_0^{-1}\ldots x_{n-1}^{-1})) \right) > \delta,$$
 for every $\underline{x} \in X_{\mathcal{G}}^+$, and $\min \{l-n, n\} > N$.
 The proof now follows from Lemma \ref{lemma:bps}.
\end{proof}
\end{lemma}

\begin{proof}[Proof of Proposition \ref{prop:poincareradioespectral}]
 Let $\mathcal{P}_n$ the set of paths $(e_1,\ldots,e_n)$ of length $n$ that are entirely contained in the recurrent piece $\mathcal{G}^r$.
 By Lemma \ref{lemma:densityandfiniteness}, the critical exponent of the Dirichlet series associated to $\varphi$ is the same as the critical exponent of the series
   $$\sum_{n \geq 0} \sum_{(e_1,\ldots,e_n) \in \mathcal{P}_n} e^{-s \varphi (\kappa (ev(e_1,\ldots,e_n)^{-1}))},$$
 because each element in the image of the evaluation of finite paths is represented by at most finitely many paths, and such elements are $c-$dense in the group.

 This Dirichlet series converges (diverges) if 
 $$\limsup_n \frac{1}{n} \log \left( \sum_{(e_1,\ldots,e_n) \in \mathcal{P}_n} e^{-s \varphi (\kappa (ev(e_1,\ldots,e_n)^{-1}))} \right) <0,$$
 ($>0$ respectively).

 On the other hand, notice that $\Sigma_n B (\underline{x}) = \sigma_\Theta (\rho (x_0^{-1}\ldots x_n^{-1}), \xi \circ p (T^n (\underline{x})))$ by the cocycle property.
 Suppose that $P(-s B_\varphi) <0$, then

   $$\limsup_n \frac{1}{n} \log \left( \sum_{|w| = n} \exp\left(\sup_{\underline{x} \in [w]}\Sigma_n B_\varphi(\underline{x})\right) \right) <0,$$

 In particular, the previous lemma states that when this happens, the condition of the previous paragraph is satisfied, so the series converges.
 The same argument shows that $P(-s B_\varphi) >0$ implies that the Dirichlet series diverges, and hence $h_\rho(\varphi) > s$.

 The variational principle (see \cite[Theorem 10.4.1]{viana2016foundations}) implies that the function $s \mapsto P(-s B_\varphi)$ is monotonic, particularly, there can be at most one solution to $P(-sB_\varphi) =0$.
\end{proof}

\begin{rmk}
A result by Quint (see Section 7 of \cite{quint2002patterson}) shows one can compute the critical exponent on the symmetric space 
  $$h^\mathbb{X}_\rho = \limsup_{t \to \infty}\frac{1}{t}\log\#\{\gamma \in \Gamma: \|\kappa(\rho(\gamma))\| \le t\}$$
   as the critical exponent of a particular functional $\varphi_0 \in \mathfrak{a}^*$ (depending on the representation). Thus, Theorem \ref{thm:general+} also provides control on this exponent. 
\end{rmk}

%% file: sections/proof.tex
\section{Separation and its consequences}
\label{sec:proof}

In this section, we define separation conditions and explain how they allow us to control the length spectrum and critical exponents of Anosov representations.

For the rest of the section, $G$ will be a real semisimple Lie group with non-compact factors, $o \in \mathbb{X}$ a basepoint with a preferred choice of Weyl chamber $\mathfrak{a}^+$ for $\mathfrak{a}$ a Cartan subspace tangent to $o$.
The basepoint induces a distance in all partial flag manifolds as described in Section \ref{subsec:symmetric}.
The separation conditions are defined as follows:


\begin{defi}
   \label{def:separatednested}
 Given $\rho : \Gamma \rightarrow G$ a representation and $\mathcal{G}$ a strong Markov structure for $(\Gamma,S)$, we say that $\rho$ is
    \begin{enumerate}
        \item \emph{$\varepsilon$-separated} with respect to $\Theta$ and $\mathcal{G}$ if there exists $(M_v)_{v \in \mathcal{V}}$ an invariant family of multicones for which $M_v$ is contained in $\mathcal{U}^- (\rho (\pi (e)^{-1}),\varepsilon)$, the complement of the $\varepsilon$-neighborhood of $\rho(\pi(e)^{-1})^-$, for every edge $e$ directed towards $v \in \mathcal{V}$.
        \item \emph{$\varepsilon$-strongly separated}, if it is $\varepsilon$-separated, and, for every $g \in S$, $\rho (g)$ is an $\varepsilon-$separated $\Theta$ loxodromic.
    \end{enumerate}
\end{defi}

It follows from Theorem \ref{thm:coding} that every $\varepsilon$-separated representation is $\Theta-$Anosov, and any $\Theta-$Anosov representation is separated for some $\varepsilon >0$ and some $\mathcal{G}$.

Recall that the cone $\mathfrak{a}_\Theta^+ \subset \mathfrak{a}_\Theta$ induces a partial order $\leq_\Theta$ in $\mathfrak{a}_\Theta$.
Moreover, define $R_\Theta \in \mathfrak{a}^+_\Theta$ via the condition that $\overline{\omega_\alpha} (R_\Theta) = 1$ for every $\alpha \in \Theta$.
We can rephrase Proposition \ref{coro:estimatesgeneralcase} as giving us estimates for the Busemann cocycle as introduced in Definition \ref{defi:busemannpotential}

\begin{lemma}
\label{lemma:key}
 Let $\rho: \Gamma \rightarrow G$ be a representation, and $R_\varepsilon \in \mathfrak{a}_+$ defined as $R_\varepsilon = - (\log \sin \varepsilon) R_\Theta$.  
 Then the following hold:
    \begin{enumerate}
        \item If $\rho$ is $\varepsilon$- separated with respect to $\mathcal{G}$ and $\Theta$, then given $J : X_\mathcal{G} \rightarrow \mathfrak{a}^+_\Theta$ the $1$-locally constant potential $J (\underline{x}) = \lambda_\Theta (\rho (x_0)^{-1})$, we have $$J - R_\varepsilon \leq_\Theta B.$$
        \item If $\rho$ is $\varepsilon$-strongly separated with respect to $\mathcal{G}$ and $\Theta$, then defining $C : X_\mathcal{G} \rightarrow \mathfrak{a}$ as $C (\underline{x}) =  \kappa_\Theta (\rho (x_0)^{-1})$, then $$C -  2R_\varepsilon \leq_\Theta B.$$
    \end{enumerate}
    \begin{proof}
 Both proofs are identical, so we only prove the first inequality.
 Let $p : X_\mathcal{G} \rightarrow \partial \Gamma$, the coding coming from the sofic shift, and $\xi : \partial \Gamma \rightarrow \mathcal{F}_\Theta$ the limit map of $\rho$.
 It follows from Theorem \ref{thm:coding} that $\xi \circ p (T (\underline{x}))$ is contained in a multicone of the vertex of $\mathcal{V}$ which has $x_0$ as a label of an outgoing edge.
 Particularly, $\xi \circ p (T (\underline{x})) \in \mathcal{U}^- (\rho (x_0)^{-1},\varepsilon)$ by the separation condition.
 This allows us to use Proposition \ref{coro:estimatesgeneralcase} to get the desired inequality.
\end{proof}
\end{lemma}

\begin{rmk}
\label{rmk:upperbound}
 Notice that $B \leq_\Theta C$ regardless of any separation condition.
\end{rmk}

The previous lemma is the key estimate for our main theorems.
\subsection{Consequences of separation}

In this section, we exhibit a series of consequences of Lemma \ref{lemma:key}.
We start by observing that such estimates yield control over the length spectrum.
Whenever a system of generators is chosen on $\Gamma$, we will denote by $|\cdot|$ its induced word metric on $\Gamma$.

\begin{prop}
\label{prop:lengthcomparison}
 Let $\rho$ be a $\varepsilon$-separated representation with respect to $\mathcal{G}$ and $\Theta$, and let $\gamma \in \Gamma$ be represented by a closed loop in the geodesic automaton associated with $(\Gamma, S)$ with labels $g_1,\ldots, g_l$, then 
   $$\sum_{i=1}^l \lambda_\Theta (\rho(g_i)) - |\gamma|_S R_\varepsilon \leq_\Theta \lambda_\Theta (\rho(g)).$$
 Moreover, if $\rho$ is strongly separated, then
   $$\sum_{i=1}^l \kappa_\Theta (\rho (g_i)) - 2|\gamma|_S R_\varepsilon \leq_\Theta \lambda_\Theta (\rho (g)) \leq_\Theta \sum_{i=1}^l \kappa_\Theta (\rho (g_i)).$$
\begin{proof}
 Recall that given $g \in G$ a $\Theta-$loxodromic element, then $\lambda_\Theta (g) = \sigma_\Theta (g,g_+)$.
 Particularly, notice that if a closed loop in the automaton represents $g$, it corresponds to an element $\overline{x}$ of $X_\mathcal{G}$ that is periodic for the shift map, with period equal to $|\gamma|_S$, and $\Sigma_n B (\underline{x}) = \lambda_\Theta (\gamma^{-1})$ by the cocycle condition.
 Taking Birkhoff sums on the estimate given by the first or second item of Lemma \ref{lemma:key} yields the desired result.
   \end{proof}
\end{prop}
Given $\varphi \in \mathfrak{a}_\Theta^*$ positive in the limit cone, use the following notations: $B_\varphi = \varphi \circ B$, $J_\varphi = \varphi \circ J$, and $C_\varphi = \varphi \circ C$.

\begin{defi}
 Given $S$ a system of generators in $\Gamma$, $\rho : \Gamma \rightarrow G$ a $\Theta-$Anosov representation, and $\varphi \in (\mathfrak{a}_\Theta)^*$ positive on the limit cone of $\rho$.
 We define 
   \begin{itemize}
    \item the \emph{approximating Jordan exponent} $h^\lambda_\rho (\varphi)$ as the number such that the pressure $P(- h^\lambda_\rho (\varphi) J_\varphi) = 0$, for $J$ the potential defined in Lemma \ref{lemma:key}.
    \item the \emph{approximating Cartan exponent} $h^\kappa_\rho (\varphi)$ is the number such that \newline $P(- h^\kappa_\rho (\varphi) C_\varphi) = 0$.
   \end{itemize}
\end{defi}

The approximating exponents are much easier to compute than the critical exponents. Since the Cartan and Jordan cocycles are locally constant, this amounts to computing the spectral radius of a specific matrix (see Appendix \ref{app:subshifts}). The following proposition tells us how good the approximation of these exponents is.

\begin{prop}
\label{prop:criticalexponents}
 Let $\rho$ be a $\varepsilon$-separated representation with respect to $\mathcal{G}$ and $\Theta$, and $\varphi \in (\mathfrak{a}_\Theta)^*$ verifying that $m^\lambda_{\rho,\varphi} := \min_{g \in S} \varphi (\lambda (\rho (g))) > \varphi (R_{\varepsilon})$, we get the upper bound
   $$h_\rho (\varphi) \leq \left( 1 - \frac{\varphi (R_{\varepsilon})}{m^\lambda_{\rho,\varphi}} \right)^{-1} h_\rho^\lambda.$$
 Moreover, if $\rho$ is $\varepsilon$-strongly separated and $m^\kappa_{\rho,\varphi} := \min_{g \in S} \varphi (\kappa (\rho (g))) > 2 \varphi (R_{\varepsilon})$, then:
   $$h_\rho^\kappa (\varphi) \leq h_\rho (\varphi) \leq \left( 1 - \frac{2\varphi (R_{\varepsilon})}{m^\kappa_{\rho,\varphi}} \right)^{-1} h^\kappa_\rho (\varphi).$$
\end{prop}

\begin{proof}
Notice that Proposition \ref{prop:lengthcomparison} implies that $\varphi$ is positive on the limit cone of $\rho$.
Recall from Proposition \ref{prop:poincareradioespectral}, that $h_\rho (\varphi)$ verifies that $P(- h_\rho (\varphi)B_\varphi) = 0$. 
We will proceed to prove the second chain of inequalities, as the same argument shows the first one.

Lemma \ref{lemma:key} implies
$$-t C_\varphi \leq -t B_\varphi \leq -t \left( 1 - \frac{2\varphi (R_\varepsilon)}{m_{\rho,\varphi}^\kappa}\right) C_\varphi.$$
Notice that the right-hand side is positive by hypothesis.
Using the left inequality and monotonicity of the pressure implies that since $P(-h_\rho^\kappa (\varphi) C_\varphi) = 0$, then $P(-h_\rho^\kappa (\varphi) B_\varphi) \geq 0$.
Thus, $h_\rho^\kappa (\varphi) \leq h_\rho (\varphi)$ by Proposition \ref{prop:poincareradioespectral}.
The same argument applied to the inequality on the right leads to the desired result.
\end{proof}

We hope this separation notion can be used in future computational work on Anosov representations to approximate critical exponents.
An answer to the following question would be useful in this direction:

\textbf{Question:} Given $\rho : \Gamma \rightarrow G$ a $\Theta-$Anosov representation, is there always a strong Markov structure $\mathcal{G}$ such that $\rho$ is $\varepsilon$-strongly separated for some $\varepsilon$? 

Such a generating system would allow us to use the last inequality in the previous proposition to estimate the critical exponent.
The more ambitious question would be to ask whether there are always generators $S_n$ of $\Gamma$, making the error in the approximating exponent go to zero.
Such an inductive procedure would be similar in nature to the algorithm proposed in \cite{mcmullen1998hausdorff} to approximate critical exponents in rank one.

\subsection{Strongly separated diverging representations}

Notice that fixing $\varepsilon$, the bounds on the previous section get better when the Jordan/Cartan projection of the generators gets big.
This section explores some consequences of that observation.

Given $(\rho_n) \in Hom (\Gamma, G)^\mathbb{N}$ a sequence of representations, we say that they form a \emph{(strongly) separated sequence} with respect to $\Theta \subset \Pi$ and a strong Markov coding $\mathcal{G}$ if there exists $(\varepsilon_n)$ a sequence such that $\rho_n$ is $\varepsilon_n$-(strongly) separated with respect to $\mathcal{G}$ and $\Theta$.

Putting together the results from the previous section, we get that the strongly separated conditions give us control over the critical exponent and the Thurston asymmetric length for such representations.

Let $G$ a connected, simple and center-free Lie group, $\Theta \subseteq \Pi$, and $\mathcal{X}$ a subset of Zariski-dense, $\Theta$-Anosov representations.
Suppose $\varphi \in \mathfrak{a}^*_\Theta$ is positive on the limit cone of each representation in $\mathcal{X}$. Assume also that every automorphism of $G$ leaving $\varphi$ invariant is inner. We define the \emph{Thurston asymmetric metric} as 
$$d^\varphi_{Th}(\rho_1, \rho_2) = \log\left(\sup_{[\gamma] \in [\Gamma]}\frac{h_{\rho_2}(\varphi)L^\varphi_{\rho_2}(\gamma)}{h_{\rho_1}(\varphi)L^\varphi_{\rho_1}(\gamma)}\right)$$

This notion was introduced in \cite{carvajales2024thurston}, where it was shown to be an asymmetric metric. We are now ready to state the main theorem of the paper.

\begin{thm}
\label{thm:general+}
Let $\Theta \subseteq \Pi$ with, $\varphi \in \mathfrak{a}_\Theta^*$, and $(\rho_n) \in Hom (\Gamma,G)^\mathbb{N}$ be a $\Theta-$strongly separated sequence of representations with respect to a strong Markov coding $\mathcal{G}$ such that $\frac{\min_{g \in S} \varphi (\kappa (\rho_n (g)))}{ -\log \sin \varepsilon_n} \to \infty,$ with $(\varepsilon_n)$ a sequence as above. 
Then
$$\lim_{n \to \infty} \frac{h_{\rho_n} (\varphi)}{h_{\rho_n}^\kappa (\varphi)} = 1.$$
Moreover, assuming $\iota(\Theta) = \Theta$ and that $0 \neq \inf_n h_{\rho_n}^\kappa \varphi (\kappa (\rho_n (g)))$ for any $g \in S$, then
$$\sup_n d_{Th}^\varphi (\rho_n, \rho_0) <\infty,$$
and similarly, if $\sup_n h_{\rho_n}^\kappa \varphi (\kappa (\rho_n (g))) <\infty$ for every $g \in S$, then
$$\sup_n d_{Th}^\varphi (\rho_0,\rho_n) < \infty.$$
\end{thm}

\begin{rmk}
   Notice that both conditions are ensured if $$\max_{g \in S} \varphi (\kappa (\rho (g)))/ \min_{g \in S} \varphi (\kappa (\rho (g)))$$ is bounded away from zero and infinity for all $n$, as stated in the introduction.
\end{rmk}

For the proof, we will need the following consequence of Proposition \ref{prop:lengthcomparison}:

\begin{lemma}
   \label{lemma:lasdistancias}
 Let $(\rho_n) \in Hom (\Gamma,G)^\mathbb{N}$ be a strongly separated sequence of representations, and $\varphi \in \mathfrak{a}^*$ verifying the conditions of Theorem \ref{thm:general+}.
 Then given $\gamma \in \Gamma$ represented by a closed loop in the geodesic automaton representing the word $\gamma = g_1\ldots g_l$, then
   $$\lim_{n \to \infty} \frac{\varphi(\lambda_\Theta (\rho_n (\gamma)))}{\sum_{i=1}^l \varphi (\lambda_\Theta (\rho_n (g_i)))} =1,$$
 moreover, the convergence is uniform on $\gamma$.
\end{lemma}

\begin{proof}
 Dividing the inequality provided by Proposition \ref{prop:lengthcomparison} by $\sum_{i=1}^l \varphi (\kappa (\rho_n (g_i)))$, we get that      $$\vline \frac{\varphi (\lambda_\Theta (\rho_n (\gamma)))}{\sum_{i=1}^{|\gamma|} \varphi (\kappa (\rho_n (g_i)))} -1 \vline \leq \frac{l \varphi (R_{\varepsilon_n})}{\sum_{i=1}^l \varphi (\kappa (\rho_n (g_i)))} \leq \frac{\varphi (R_{\varepsilon_n})}{\min_{g \in S} \varphi (\kappa (g))}.$$
 Notice that our hypothesis implies that the right-hand side goes to zero with $n$ and does not depend on $\gamma \in \Gamma$. 
 This shows the result after noticing that $$\lim_{n \to \infty} \frac{\varphi (\kappa (\rho_n (g)))}{\varphi (\lambda (\rho_n (g)))} = 1$$ for any $g \in S$.
\end{proof}

\begin{proof}[Proof of Theorem \ref{thm:general+}]
The convergence of the critical exponent $h_{\rho_n} (\varphi)$ to $h_{\rho_n}^\kappa$ follows immediately from Proposition \ref{prop:criticalexponents} and the observation that the error term goes to zero by our assumptions.

To prove the second part, we will begin by controlling the quotients 
$$\frac{h_{\rho_n} (\varphi) L_{\rho_n}^\varphi ([\gamma])}{h_{\rho_0} (\varphi) L_{\rho_0}^\varphi ([\gamma])},$$
for $[\gamma] \in [\Gamma]$ a conjugacy class.
Notice that such a quotient is independent of whether we take the conjugacy class of an element or the conjugacy class of a positive power.
By Proposition \ref{prop:cantrell}, there exists $\gamma'$ and $N \in \ZZ$ such that $[\gamma^N] = [\gamma']$, such that $\gamma' = g_1 \ldots g_l$ is represented by a closed loop in $\mathcal{G}^r$. 

First, we assume $N > 0$.
By definition, the stable norm is given by $|\gamma|_\infty = \lim_{n \to \infty} \frac{|\gamma^n|}{n}$. We have that $|\gamma^N|_\infty = N|\gamma|_\infty = l.$
By the previous lemma and the first item, we get that given $\varepsilon >0$, there is $n$ large enough for which
$$1- \varepsilon \leq \frac{h_{\rho_n} (\varphi) \varphi (\lambda_\Theta (\gamma'))}{h_{\rho_n}^\kappa (\varphi) \sum^l_{i=1} \varphi (\kappa (\rho_n (g_i))) } \leq 1 + \varepsilon,$$
for some $\gamma$ as chosen in the paragraph above.
Moreover, this $n$ is uniform in $\gamma$.

On the other hand, since $\rho_0$ is Anosov, we get from Kassel-Potrie \cite{fanny2022} that there exists $A,B>0$ for which:
$$A |\gamma'|_\infty - B \leq \varphi (\lambda_\Theta (\rho_0 (\gamma'))) \leq A |\gamma'|_\infty + B,\: \forall \gamma \in \Gamma.$$

Putting everything together, we get that
$$\frac{h_{\rho_0} (\varphi) L_{\rho_0}^\varphi ([\gamma])}{h_{\rho_n} (\varphi) L_{\rho_n}^\varphi ([\gamma])} =  \frac{h_{\rho_0} (\varphi) L_{\rho_0}^\varphi ([\gamma'])}{h_{\rho_n} (\varphi) L_{\rho_n}^\varphi ([\gamma'])} \leq (1-\varepsilon)^{-1} \frac{h_{\rho_0} (\varphi) (A|\gamma'|_\infty + B)}{h_{\rho_n}^\kappa (\varphi) |\gamma'|_\infty \min_{g \in S} \varphi (\kappa \rho_n (g))}.$$
Notice that under the condition $0 \neq \inf_n h_{\rho_n}^\kappa \varphi (\kappa (\rho_n (g))) = K$, we get
$$\frac{h_{\rho_0} (\varphi) L_{\rho_0}^\varphi ([\gamma])}{h_{\rho_n} (\varphi) L_{\rho_n}^\varphi ([\gamma])} \leq \frac{h_{\rho_0} (\varphi) (A + B)}{K(1-\varepsilon)}.$$
This is a uniform bound for every homotopy class $[\gamma] \in [\Gamma]$ and $n$ large enough, which proves the claim.

If $N < 0$, we get that $$L^\varphi_{\rho_n}(\gamma) = \frac{L^{\varphi \circ \iota}_{\rho_n}(\gamma'^{-N})}{-N},$$
where $\iota$ is the opposition involution. Because we assumed $\Theta = \iota(\Theta)$, we get that $\varphi \circ \iota$ is well-defined in $\mathfrak{a}_\Theta$ and positive in the limit cone, so $\varphi \circ \iota$ lies in the hypothesis of Lemma \ref{lemma:lasdistancias}, and the same calculation yields the same result.

Finally, notice that
$$\frac{h_{\rho_n} (\varphi) L_{\rho_n}^\varphi ([\gamma])}{h_{\rho_0} (\varphi) L_{\rho_0}^\varphi ([\gamma])} \leq (1 + \varepsilon) \frac{h_{\rho_n}^\kappa (\varphi) |\gamma|_\infty \max_{g \in S} \varphi (\kappa (\rho_n (g)))}{h_{\rho_0} (\varphi)(A |\gamma|_\infty - B)},$$
where up to changing $[\gamma]$ to $[\gamma^n]$, we can always assume that the right-hand side is positive. 
The last claim follows similarly to the previous case.
\end{proof}

For examples of strongly divergent representations, one can reinterpret these results geometrically using equivariant graphs on the symmetric space.

\begin{rmk}
Suppose $\Gamma$ is the free group on $k$ generators, and $S$ is the standard symmetric set of generators.
In this case, given our basepoint $o \in \mathbb{X}$ and $\rho$ a representation, one can embed the Cayley graph $Cay (\Gamma, S)$ as a $\rho-$equivariant tree in the symmetric space $T \subset \mathbb{X}$ in such a way that its vertex set coincides with the orbit of $o$ with geodesic edges.

Such trees inherit a metric from the symmetric space, or more generally, an asymmetric metric from a functional $\varphi \in \mathfrak{a}_\Theta^*$ by declaring $d_\varphi (go,g'o) = \varphi (\kappa (g^{-1}g'))$ for neighboring vertices.
Now given $\rho_n$ a sequence of strongly separated representations verifying the conditions of the previous theorem, then the length of a loop $(T_n/\rho_n (\Gamma), d^\varphi)$ in the homotopy class on $\gamma \in \Gamma$ approximates the length $\varphi (\lambda_\Theta (\rho (\gamma)))$, and correspondingly, the critical exponent is asymptotic to those of $\rho$.
We can think of this graph as a ``spine'' for $\mathbb{X}/\rho_n (\Gamma)$, and should remind the reader of the classical spinal graph construction for non-compact hyperbolic surfaces (see \cite{bowditch1988natural}).
\end{rmk}

Notice that the conditions of our second item are met whenever the potentials $h_{\rho_n} (\varphi) B_{\rho_n,\Theta}$ converge to a limiting potential in $C(X_\mathcal{G})$.
Thinking in terms of the geometric interpretation explained above, this case corresponds to the situation in which the spinal graph obtained by rescaling by the critical exponent converges to another metric graph.
It is natural to think of such paths in the character variety as being incomplete, and their completion to correspond to the limiting rescaled graph.

Under the strong separation conditions, we also get good control over the limit cone of the representations.
The following is a direct consequence of Lemma \ref{lemma:lasdistancias}.

\begin{coro}
\label{coro:limitcones}
Let $(\rho_n) \in Hom (\Gamma, G)^\mathbb{N}$ be a strongly separated sequence of representations with respect to $\Theta \subseteq \Pi$ and a strong Markov coding $\mathcal{G}$, such that $\iota(\Theta) = \Theta$. Assume that $\frac{\min_{g \in S} \alpha (\kappa (\rho_n(g)))}{-\log \sin \varepsilon_n} \to \infty$ for every $\alpha \in \Theta$.
Under these conditions, given $\mathcal{C}^\Theta_n \subset \mathfrak{a}_\Theta$ the cone generated by $\kappa_\Theta (\rho_n (g))$, for $g \in S$, we get that:
$$d_{H} (p_\Theta (\mathcal{L}_{\rho_n}), \mathcal{C}^\Theta_n) \to_n 0,$$
for $d_{H}$ the Hausdorff distance in $\mathfrak{a}_\Theta$.
\end{coro}

Even though our definition of strongly separated representation gives us the freedom to choose a set of generators $S$ (however complicated), we can prove that the hypothesis of Theorem \ref{thm:general+} forces $\Gamma$ to be (virtually) free.

\begin{prop}
\label{prop:onlyfree}
 Let $\rho_n : \Gamma \rightarrow G$ be a strongly separated sequence of representations with respect to a strong Markov coding $\mathcal{G}$, and $\Theta$, such that $ \frac{-\log\sin\epsilon_n}{\min_{g \in S} \varphi (\kappa (\rho(g)))} \to 0$ for $\varphi \in \mathfrak{a}_\Theta^*$, then $\Gamma$ is virtually free.
\end{prop}

\begin{proof}
Let $\mathcal{G}^r$ be the recurrent part of the strong Markov coding. 
Since $\mathcal{G}$ is finite, there exists $p > 0$ such that every path 
of length greater than $p$ eventually enters $\mathcal{G}^r$. In particular, 
every boundary point of $\Gamma$ is the limit of a geodesic ray that is 
eventually contained in $\mathcal{G}^r$, so the shadows of paths in 
$\mathcal{G}^r$ cover $\partial\Gamma$.

Let $(e_0,\ldots,e_k)$ be a path in $\mathcal{G}^r$ starting at $v_0$. 
We define 
$$C(e_0,\ldots,e_k) = \rho'_n(\pi(e_0)^{-1}\ldots\pi(e_k)^{-1})M_{v_k}.$$
Observe that 
$$C(e_0,\ldots,e_k) = \rho'_n(\pi(e_0)^{-1}) C(e_1,\ldots,e_k).$$
By the invariance of the multicone family, $C(e_1,\ldots,e_k) \subset M_{v_1}$, 
and by the separation condition, $M_{v_1} \subset \mathcal{U}^-(\rho'_n(\pi(e_0)^{-1}),\varepsilon_n)$.
Lemmas~\ref{lemma:proximalestimate} and~\ref{lem:implicaproximal} then imply, applying the estimate $k$ times, that 
$$\diam C(e_0,\ldots,e_k) \leq Ae^{-k(m_{\rho_n,\alpha} - \log\sin\varepsilon_n)},$$
for the constant $A = \max_{v \in \mathcal{V}} \diam M_{v}$, which is uniformly bounded on $n$, and independent of the path.

Let $\delta(\Gamma, S)$ be the exponential growth rate of $\Gamma$ with respect to $S$. By Lemma~\ref{lemma:densityandfiniteness}, the number of paths of length $k$ 
in $\mathcal{G}^r$ is bounded by $Ke^{k\delta(\Gamma,S)}$, where $K$ is a constant that depends on the number of vertices in the strong Markov coding, and on the constant of density of the recurrent graph.
Fix $n$ large enough so that $m_{\rho_n,\alpha} - \log\sin\varepsilon_n > \delta(\Gamma,S)$.
For such $n$, the total diameter of all shadows of length $k$ is bounded by
$$Ke^{k\delta(\Gamma,S)} \cdot Ae^{-k(m_{\rho_n,\alpha} - \log\sin\varepsilon_n)} \to 0$$
as $k \to \infty$. Since the shadows $C(e_0,\ldots,e_k)$ cover $\partial\Gamma$ 
and their total diameter goes to zero, any connected subset of $\partial\Gamma$ 
of positive diameter leads to a contradiction. Hence $\partial\Gamma$ is totally 
disconnected, and by Stallings' theorem~\cite{stallings1968torsion}, $\Gamma$ is virtually free.
\end{proof}

The proof also shows that the limit set is supported in a ball around finitely many flags for sufficiently large $ n$.

We wrap up the section by remarking that some partial results can be obtained under (non-strong) separation hypotheses on a sequence.

\begin{prop}
\label{thm:nested}
 Let $(\rho_n) \in Hom (\Gamma ,G)^\mathbb{N}$ be a sequence of separated representations with respect to $\Theta$ and $\mathcal{G}$, and $\varphi \in (\mathfrak{a}_\Theta)^*$ satisfying that $\frac{-\log \sin \varepsilon_n}{\min_{g \in S} \varphi (\lambda (g))} \to 0$.
 Under these conditions, we get that for any $\delta >0$, there exists $n_0$ such that
 $$h_{\rho_n} (\varphi) \leq (1 + \delta) \frac{C(\mathcal{G})}{\min_{g \in S} \varphi (\lambda (\rho (g)))},$$
 for every $n \geq n_0$, where $C(\mathcal{G})$ is a constant depending on the combinatorics of $\mathcal{G}$.
 \begin{proof}
 Proposition \ref{prop:criticalexponents} implies
   $$\limsup_n \frac{h_{\rho_n} (\varphi)}{h_{\rho_n}^\lambda (\varphi)} \leq 1.$$
 Recall that the approximating Jordan exponent can be computed in terms of the Jordan potential $J_\varphi$.
 Bounding $J_\varphi$ below by the constant potential defined by $\min_{g \in S} \varphi (\lambda (g))$, we get the desired result.
 \end{proof}
\end{prop}

\begin{rmk}
 Notice that the hypotheses of the previous proposition are not enough to guarantee separation.
 This can be seen with the matrices
    $$g_n = 
    \begin{pmatrix}
 n^2 & 0 & 0\\
    0 & n & n^2 \\
    0 & 0 & \frac{1}{n^3}
    \end{pmatrix}.$$
 Notice that the attracting eigenline is $[e_1]$, the repelling hyperplane of $g_n$ is given by $x = 0$, however $g_n (1:0:1) = (1:1:1/n^5) \to (1:1:0) \neq (1:0:0)$.
 Moreover, one can check that $\alpha_1 (\lambda (g_n)) \to \infty$, but $\alpha_1 (\kappa (g_n))$ stays bounded. 
\end{rmk}

We do not know whether the analogous restriction ($\Gamma$ virtually free) holds in this situation as well.

%% file: sections/sl3.tex
\section{Application to convex projective structures on a pair of pants}
\label{sec:sl3}
The goal of this section is to apply the main theorem of the paper in a concrete example. We explain the Fock Goncharov parametrization of convex projective structures on a pair of pants, and we estimate the degeneration of the critical exponent along families of representations.

\subsection{Projective invariants}
The space of $PGL_2(\mathbb{R})$-orbits of ordered $4$-tuples of pairwise distinct points in $\mathbb{P}^1$ is one-dimensional and can be parametrized with the cross-ratio. 
The normalization 
$$CR(\infty,-1,0,c)=c$$
uniquely defines this invariant, where $\infty, -1,0,c$ denotes the given $4$-tuple of points in a chosen affine chart.
Particularly, notice that the cross-ratio of $(A, B, C, D)$ is positive if and only if the points are cyclically ordered.
In higher-dimensional projective spaces, this defines an invariant of $4-$tuples of aligned points.

We will be interested in constructing invariants of configurations of flags $F = (p,\ell) \in \mathbb{P}^2 \times (\mathbb{P}^2)^*$ in $\mathbb{R}^3$.
Recall that a flag can be represented by a pair $(v,\alpha) \in \mathbb{R}^3 \times (\mathbb{R}^3)^*$ such that $\alpha (v) = 0$, and this representation is well-defined up to rescaling on each factor.
We say that two flags $F_1, F_2$ are in general position if $p_1,p_2$ and $\ell_1 \cap \ell_2$ are in general position.

The set of $SL_3(\RR)$ orbits of triples of flags in pairwise general position can be parametrized by one single invariant, called the \emph{triple ratio}, defined via
$$TR (F_1,F_2,F_3) = \frac{\alpha_1 (v_2)\alpha_2 (v_3)\alpha_3 (v_1)}{\alpha_1 (v_3)\alpha_2 (v_1)\alpha_3 (v_2)},$$
where $(v_i,\alpha_i)$ represent $F_i$.
This invariant measures how far away the three lines joining $p_i$ with $\ell_{i-1} \cap \ell_{i+1}$ are from being coincident at a point.

\subsection{Fock-Goncharov parametrization}

The goal of the section is to review Fock-Goncharov's parametrization of the space of marked convex projective structures $\mathbb{RP} (\Sigma)$ on a non-compact orientable surface $\Sigma$ with boundary.
We closely follow the exposition of the original article \cite{fock2007moduli}. Proofs of unreferenced claims can be found there.

Sometimes, we do not normalize matrices by the determinant. This is not a problem when considering the projective action.

Recall that a convex-projective structure on a surface $\Sigma$, is a $(\mathbb{P}^2, SL_3 (\mathbb{R}))$ geometric structure on a surface in the sense of Ehresmann-Thurston (see \cite[Chapter 3]{thurston2022geometry}), verifying that the image of the developing map is a convex subset of $\mathbb{P}^2$.
We assume that each boundary component develops into a line segment.

A convex-projective structure is determined by a pair $(D,\rho)$, where $D$ is the developing map $D: \widetilde{\Sigma} \rightarrow \mathbb{P}^2$, equivariant with respect to a representation $\rho: \pi_1 (\Sigma) \rightarrow SL_3 (\mathbb{R})$ that we call the \emph{holonomy representation} (and this pair is unique up to the action of $SL_3 (\mathbb{R})$).
It follows from Fock-Goncharov that the holonomies of these structures are always Anosov representations and therefore admit a limit map $(\xi,\xi^*): \partial \pi_1 (\Sigma) \rightarrow \mathcal{F}$.
Moreover, this limit map satisfies some positivity properties that we will soon describe.

Denote by $\hat{\Sigma}$ the punctured surface obtained from $\Sigma$ by removing the boundary.
We can always choose an ideal triangulation $\mathcal{T}$ of $\hat{\Sigma}$ with the property that every edge connects punctures.
This triangulation can be realized in $\Sigma$ as a geodesic lamination $\mathcal{L}_\mathcal{T}$ after choosing a hyperbolic metric. 
To see this, tighten the edges of the triangulation into geodesics that spiral infinitely toward the boundary components. 
We assume the boundary geodesics receive an orientation from $\Sigma$, and the leaves of $\mathcal{L}$ spiral in the direction opposite to this orientation.

We will now explain how such a choice of triangulation leads to a parametrization of the space $\mathbb{RP} (\Sigma)$ using the geometry of the limit map.
Notice that the auxiliary hyperbolic structure on $\Sigma$ identifies $\partial \pi_1 (\Sigma)$ with a closed subset of $\partial \mathbb{H}^2$.

\begin{description}
\item[Triangle invariants] 
Lift an ideal triangle $t$ bounded by the lamination defined by the triangulation to $\mathbb{H}^2$ (corresponding to a face of $\mathcal{T}$).
Such a lift has vertices at three points $x_1,x_2,x_3 \in \partial \mathbb{H}^2$ ordered counter-clockwise.
Using the limit map, we can define three flags $F_i = (\xi(x_i),\xi^*(x_i))$, $i=1,2,3$ and use this triple to define an invariant of $T$ via:
$$X_t = \text{TR} (F_1,F_2,F_3).$$
This number is independent of the choice of lift and hyperbolic metric.
Moreover, Fock-Goncharov proved that $X_t$ is always strictly positive.
\item[Edge invariants] We will assign to every oriented edge $e$ of $\mathcal{T}$ a pair of invariants.
 Choose $L_e$ a leaf of $\mathcal{L}_\mathcal{T}$ corresponding to $e$, and lift it to $\mathbb{H}^2$.
 Such a lift will be adjacent to a pair of ideal triangles, and let $x_1,x_2,x_3,x_4$ be a counter-clockwise ordered tuple of points corresponding to the vertices of such triangles in such a way that $x_1$ and $x_3$ are the positive and negative endpoints of our leaf.
 This defines four flags $F_i = (\xi,\xi^*) (x_i) = (p_i,\ell_i)$, $i=1,2,3,4$ in general position.

 Given such a $4-$tuple of points, we can define a shear invariant $Z_e$ as the cross ratio of the lines $\ell_3, p_3p_2, p_3p_1$ and $p_3p_4$ (notice that these are four lines passing through a point, therefore represent four aligned points in the dual projective plane).
 Similarly, we define $W_e$ as the cross-ratio of the lines $\ell_1, p_1p_4, p_1 p_3$ and $p_1 p_2$.
 Fock-Goncharov proved that in the Hitchin component, all these numbers are positive.
\end{description}

Such invariants parametrize the space of convex projective structures $\mathbb{RP} (\Sigma)$.

\begin{thm}[{\cite[Theorem 2.7]{fock2007moduli}}]
 Let $\Sigma$ be a compact oriented surface with boundary.
 Given triangulation $\mathcal{T}$ of $\widehat{\Sigma}$ with $F$ faces, and $E$ edges, we get that the map $\mathbb{RP}(\Sigma) \rightarrow \mathbb{R}_{>0}^F \times \mathbb{R}_{>0}^{2E}$ sending a real convex projective structure to its $X_t$, $W_e$ and $Z_e$ parameters is a homeomorphism.
\end{thm}

In fact, the inverse map is completely explicit. 
We conclude the section by briefly reviewing how to recover the holonomy $\rho$ of a convex projective structure starting from the Fock-Goncharov parameters $X_t$, $W_e$, and $Z_e$, for every $t$ face of the triangulation, and $e$ and edge (see \cite[Section 5]{fock2007moduli}).

Start by picking a point inside each face $t$, and two points in each edge $e$, which we assume is an oriented edge.
Assign to the point in $t$ the number $X_t$, and to the pair of points in $e$, the pair $W_e$, $Z_e$, in such a way that $W_e$ lies after $Z_e$ with the order inherited from the orientation of $e$.
Now consider the graph dual to the triangulation, so that each edge passes through the middle of the two marked points corresponding to the $Z'$s and $W'$s.
Finally, ``blow up'' each vertex of this dual graph to obtain a triangle contained entirely within the corresponding simplex of our triangulation, ensuring each vertex of this new triangle connects to exactly one outgoing dual edge (see figure \ref{fig:fockgoncharov}).

The procedure from the last paragraph produces an embedded graph on $\Sigma$.
We will orient each triangle inside a face according to the orientation of $\Sigma$, and the remaining vertices (coming from the dual graph) in whichever way we prefer. 
We will refer to an edge inside the triangle as a $ t$-edge, and the rest as $ e$-edges.

To reconstruct the holonomy, we assign to each $t$ and $e-$edges the matrices:
$$T(X_t) = 
\begin{pmatrix}
0 & 0 & 1\\
0 & -1 & -1\\
X_t & 1 + X_t & 1
\end{pmatrix}, \: 
E(Z_e,W_e) = 
\begin{pmatrix}
   0 & 0 & Z_e^{-1} \\
   0 & -1 & 0\\
 W_e & 0 & 0
\end{pmatrix},$$
where we assume the $t-$edge is contained in the face $t$, and the oriented $e-$edge verifies that the point corresponding $Z_e$ lies to the right of this edge, and $W_e$ to its left.

Notice that the graph we just created is a spine of $\Sigma$, and is therefore homotopy equivalent to it. 
In particular, any $\gamma \in \pi_1 (\Sigma)$ can be represented by a loop in the graph passing through a favorite vertex $v_0$.
We define $\rho$ by assigning to such an element the product of matrices on the graphs that it traverses (mind the orientations).
One can show that this defines a representation $\rho$ corresponding to the holonomy of a convex projective structure.

Remarkably, one gets that
$$E(Z,W) T(X) = 
\begin{pmatrix}
 Z^{-1} X & Z^{-1} (1+X) & Z^{-1} \\
   0 & 1 & 1\\
   0 & 0 & W
\end{pmatrix}.
$$
is a positive, upper triangular matrix.
This will allow us to pin down the positions of eigenvectors and compute the eigenvalues of the holonomies of boundary elements.

\begin{lemma}
\label{lemma:eigenvalues}
 Let $\mathfrak{a} \subset \mathfrak{sl}_3 (\mathbb{R})$ be the standard Cartan subspace, and $\lambda : SL_3 (\mathbb{R}) \rightarrow \mathfrak{a}^+$ the Jordan projection.
 Then given $g = E(Z_k,W_k) T(X_k) \ldots E(Z_1,W_1) T(X_1)$, such that $\prod_{i=1}^l Z_i^{-1} X_i < 1$ and $\prod_{i=1}^l W_i > 1$, then we get that
 $$\alpha_1 (\lambda (g)) = \sum_{i=1}^l \log W_i, \: \alpha_2 (\lambda (g)) = \sum_{i=1}^l \log Z_i - \sum_{i=1}^l \log X_i .$$
 Moreover, if $W_i^n,Z_i^n$, and $X_i^n$ are sequences verifying the inequalities above, and such that $Z_i^n, W_i^n \to \infty$, $(Z_i^n)^{-1} X_i^n \to 0$, then the stable, central and unstable direction of $g_n = E(Z_k^n,W_k^n) T(X_k^n) \ldots E(Z_1^n,W_1^n) T(X_1^n)$ converge to $e_1$, $e_2$ and $e_3$ respectively. 
\begin{proof}
 The first part of the proof is a straightforward computation.
 As for the second part, notice that under these conditions, the first row of the blocks $E(Z_i^n, W_i^n) T(X_i^n)$ goes to zero, when $n \to \infty$.
 Multiplying such upper triangular matrices, we get something of the form
   $$g_n = 
   \begin{pmatrix}
 o(1) & o(1) & o(\prod_{i=1}^l W_i^n) \\
   0 & 1 & o(\prod_{i=1}^l W_i^n)\\
   0 & 0 & \prod_{i=1}^l W_i^n
   \end{pmatrix}.$$
 The stable, central, and unstable directions of these matrices are approximately $e_1$, $e_2$, and $e_3$, respectively.
\end{proof}
\end{lemma}

\subsection{Strongly separated holonomies in the pair of pants}
In this section, we specialize to the case where $\Sigma$ is a pair of pants.
The goal is to detect an explicit family of convex projective structures whose holonomy is strongly separated, as in Definition \ref{def:separatednested}.

Recall that by equipping a pair of pants with a hyperbolic metric with totally geodesic boundary, we can decompose this pants as a union of two ideal triangles, each bounded by a lamination that spirals around the cuffs in the direction opposite to their orientation (as in Figure \ref{fig:fockgoncharov}).
This triangulation parametrizes the space of convex projective structures using eight parameters: two triple ratios $X_1$ and $X_2$, and six cross ratios $ Z_1$, $ W_1$, $ Z_2$, $ W_2$, $ Z_3$, and $W_3$.

\begin{center}
    \begin{figure}[h!]
        \includegraphics[scale=1]{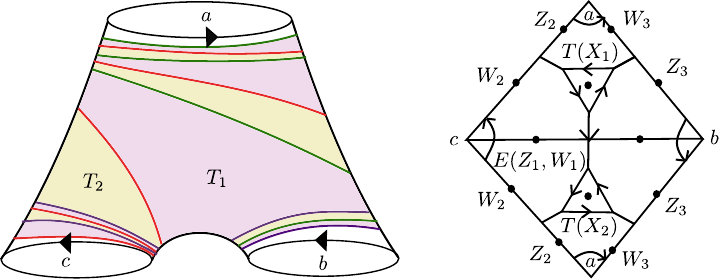}
        \caption{Lamination on the pair of pants, and the graph recovering the holonomy. Notice that the edges are identified to get a punctured sphere.}
        \label{fig:fockgoncharov}
    \end{figure}
\end{center}

Moreover, the procedure described in the previous section allows us to explicitly compute the holonomy $\rho_{\vec{X},\vec{W},\vec{Z}}$ associated with these parameters.
A quick inspection of Figure \ref{fig:fockgoncharov} shows
\begin{equation}
\begin{split}
&\rho_{\vec{X},\vec{W},\vec{Z}}(a^{-1}) = E(W_3,Z_3) T(X_2) E(Z_2,W_2) T(X_1),\\
&\rho_{\vec{X},\vec{W},\vec{Z}}(b^{-1}) = Ad (T(X_1)) (E(W_1,Z_1) T(X_2) E(Z_3,W_3) T(X_1)),\\
&\rho_{\vec{X},\vec{W},\vec{Z}}(c^{-1}) = Ad(T(X_1)^{-1}) ( E(W_2,Z_2) T(X_2) E(Z_1,W_1) T(X_1)),
\end{split}
\end{equation}
where $a,b,c$ are the cuffs of the pants.
Notice that $\pi_1 (\Sigma) \cong F_2$, and $a,b,c$ represent a set of generators verifying the relation $abc = id$ (just as in Example \ref{ex:codings}).
We will denote by $S$ the symmetric system of generators defined by $a,b$, and $c$.

Our goal is to impose conditions on the parameters that guarantee that our representation is strongly separated.

\begin{defi}
\phantom{v}
\begin{itemize}
\item We will refer to a convex projective structure on the pairs of pants as \emph{admissible} if its Fock-Goncharov parameters verify the equations
\begin{equation}
\label{eq:admissible}
\begin{split}
   & \log W_i + \log Z_{i-1} - \log X_1 + \log X_2 >0,\\
   & \log Z_i + \log W_{i-1} > 0,
\end{split}
\end{equation}   
for every $i = 1,2,3$ (indices taken mod $3$).

\item A sequence $(\rho_n)\in Hom(\pi_1 (\Sigma),SL_3 (\mathbb{R}))^\mathbb{N}$ of holonomies of convex projective structures is called \emph{shear divergent} if the left hand sides in the inequalities of Equation \ref{eq:admissible} diverge to infinity with $n$.
\end{itemize}
\end{defi}

Thanks to Lemma \ref{lemma:eigenvalues}, these conditions allow us to pin down the positions for the eigenvectors of the holonomy.
Moreover, we have the following theorem:
\begin{thm}
\label{thm:divergingpants}
 Let $(\rho_n) \in Hom(\pi_1 (\Sigma), SL_3 (\mathbb{R}))^\mathbb{N}$ be a sequence of shear diverging convex projective structures such that $0< \inf_n X_1^n < \sup X_1^n < \infty$.
 Under these conditions, if $\varepsilon>0$, there exists $n_0$ such that $\rho_n$ is  $\varepsilon$-strongly separated with respect to the system of generators $S$.
\end{thm}

We will start by showing that $\rho (s)$ is an $\varepsilon$-loxodromic element for every $s \in S$.
In $SL_3 (\mathbb{R})$, this means loxodromic respect to the whole root system.
The following lemma will be useful:

\begin{lemma}
   \label{lem:estimativasl3}
 Let $(g_n)_{n\in\mathbb N}$ be a sequence in $SL_3(\mathbb R)$ of loxodromic matrices with eigenvalues $\lambda_{1,n},\ \lambda_{2,n},\ \lambda_{3,n}$
 ordered by modulus.
 Let $\ell_{i,n}\in \mathbb P(\mathbb R^3)$ be the eigenline corresponding to $\lambda_{i,n}$.
   
 Assume that
 \[
 \frac{|\lambda_{1,n}|}{|\lambda_{2,n}|}\xrightarrow[n\to\infty]{}+\infty,
 \qquad
 \frac{|\lambda_{2,n}|}{|\lambda_{3,n}|}\xrightarrow[n\to\infty]{}+\infty,
 \]
 and that $\ell_{1,n}\to \ell_1$, $\ell_{2,n}\to \ell_2$, $\ell_{3,n}\to \ell_3$ in $\mathbb P(\mathbb R^3)$, where $\ell_1,\ell_2,\ell_3$ are pairwise distinct.
   
 Then for every sufficiently small $\varepsilon>0$ there exists $n_0\in\mathbb N$ such that for all $n\ge n_0$ the element $g_n$ is $\varepsilon$-separated for every root.
\end{lemma}

\begin{proof}
 The fundamental representation corresponding to the root $\alpha_2$ in $SL_3(\RR)$ is the representation that sends $g \to (g^{t})^{-1}$. Because $(g_n^t)^{-1}$ satisfies the same hypothesis of the lemma, it is enough to prove that the sequence is $\epsilon$-proximal. Fix $\epsilon < m/4$, where $m$ is the minimum pairwise distance between $\ell_i$ and $\ell_j$. We will prove the lemma as a consequence of Lemma \ref{lem:implicaproximal}. For this, we need to estimate $\|g_n|_{g_n^-}\|$.

 Observe that $g_n|_{g_n^-}$ can be thought of as a matrix in $SL_2(\RR)$. Without loss of generality, we may assume that the eigenvectors are $v_{2,n} = (1,0)$ and $v_{3,n} = (cos(\theta_n), sin(\theta_n))$, for $\theta_n = d(\ell_{2, n}, \ell_{3, n})$. Because the matrix whose columns are $v_{2, n}$ and $v_{3,n}$ has singular values $\sqrt{1 \pm cos(\theta_n)}$, we have that $$\|g_n|_{g_n^-}\| \le \lambda_{2,n}\sqrt{\frac{1 + \cos(\theta_n)}{1 - \cos(\theta_n)}}.$$
 The bound on $\theta_n$ implies that $\sqrt{\frac{1 + \cos(\theta_n)}{1 - \cos(\theta_n)}}$ is uniformly bounded, so for $n$ big enough we have that 
 $$\frac{\|g_n|_{g_n^-}\|}{\lambda_{1, n}} \le \sin^2 (\epsilon),$$
 which proves the lemma.
\end{proof}

\begin{coro}
  \label{coro:estanseparadas}
Under the conditions of Theorem \ref{thm:divergingpants}, we get that for every $\varepsilon$ there exists $n_0$ such that $\rho_n (s)$ is an $\varepsilon$-loxodromic element for every $n \geq n_0$.
\begin{proof}
  
 Notice that since $\rho_n (a)$ is a product of $e$ and $t-$matrix, we can use Lemma \ref{lemma:eigenvalues}, which tells us that asymptotically the eigenvectors of $\rho_n (a)$ are close to $e_1,e_2$ and $e_3$.
 Moreover, this lemma also tells us that under the shearing divergence conditions, we verify the hypothesis of the previous lemma, and therefore $\rho_n (a)$ is eventually $\varepsilon-$separated for any $\varepsilon$.
   
 Notice that $\rho_n (b)$ and $\rho_n (c)$ are conjugates to a product of $t$ and $e$-matrices, via a matrix $T(X_1^n)$ that stays in a compact set of $SL_3 (\mathbb{R})$ (because $X_1^n$ is bounded away from zero and infinity), the same argument shows that $\rho_n (b)$ and $\rho_n (c)$ are $\varepsilon-$separated for $n$ large enough.
\end{proof}
\end{coro}
Therefore, to prove Theorem \ref{thm:divergingpants}, it will suffice to exhibit a family of multicones.

\begin{figure}[h!]
   \begin{center}
      \includegraphics[scale=1.5]{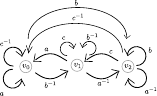}
   \end{center}
   \caption{Recurrent part of a Strong Markov structure for the free group. In the notation of the proof of Theorem \ref{thm:divergingpants}, we have that $v_0 = v(a) = v(c^{-1})$.}
   \label{fig:markov_recurrent}
\end{figure}

\begin{proof}[Proof of Theorem \ref{thm:divergingpants}]

 Recall from the notation of Section \ref{sec:epsilon}, that given $g \in SL_3(\RR)$, $\mathcal{U}^+(g, \epsilon)$ is the open ball of radius $\varepsilon$ in $\PP^2$ around $g_+$, and $\mathcal{U}^-(g, \epsilon)$ is the set of points in $\PP^2$ that are at distance at least $\epsilon$ from $g_-$.

 Consider the strong Markov coding of $\Gamma$ with respect to the set of generators $S = \{ a, b, c, a^{-1}, b^{-1}, c^{-1} \}$, as shown in Example \ref{ex:codings} (we include Figure \ref{fig:markov_recurrent} for the reader's convenience). 
 Given the metric on $\mathcal{F}$ induced from a chosen inner product in $\mathbb{R}^3$, we will define the multicone of $v_i$ to be
 $$M_{v_i} = \bigcup_{e \in \mathcal{E}: f(e) = {v_i}}\mathcal{U}^+(\rho_n (\pi(e)^{-1}), \varepsilon).$$
We need to show that this is an invariant family of multicones. Because of the Markov structure's symmetry, we restrict our study to a specific vertex.

\textbf{Step 1:} Start by noticing that thanks to Corollary \ref{coro:estanseparadas}, for any $\varepsilon > 0$ small enough, and $n$ big enough we have that $\rho_n(g)$ is $\varepsilon$-separated for any $g \in S$.

\textbf{Step 2:} We will now locate the positions of the eigendirections for the generators in order to verify the separation condition. 

A feature of this Markov coding is that the terminal vertex $f(e)$ of an edge is determined by its label $\pi (e)$.
For a generator $g \in S$, we denote by $v(g)$ the vertex that has incoming edges of label $g$.

 It follows from the proof of the previous lemma that for large enough $n$, the attracting, central, and repelling lines of $\rho_n (a)$ are close to $e_1, e_2$, and $e_3$, respectively.
 Similarly, the attracting, central and unstable directions of $\rho_n(b)$ and $\rho_n (c)$ are close to the vectors $(e_3, (1 + X_1^n)e_3 - e_2, e_1-e_2 + e_3)$ and $(e_1 - e_2+e_3, (X_1^n+1)e_1 - X_1^n e_2, e_1)$ respectively.

 A quick check on the positions of the attractors and repellors of the involved object tells us that, for $\varepsilon$ small enough and $n$ big enough, we have the inclusion 
 $$\mathcal{U}^+(\rho_n(g^{-1}), \varepsilon) \subseteq \mathcal{U}^-(\rho_n(a^{-1}),\varepsilon),$$
 for $g = a, b, b^{-1}, c^{-1}$. 
 This is the situation for $v = v(a) = v(c^{-1})$, a symmetry argument shows that the same holds for the remaining vertices.

 If we can prove the invariance, this property implies separation.
 
\textbf{Step 3:} Now we are ready to prove the invariance.
Once we do this, the strong separation follows from the previous steps.

Since every generator is $\varepsilon$-separated, the previous inclusion implies that
$$\rho_n(a^{-1})M_{v(a)} \subseteq \mathcal{U}^+(\rho_n(a^{-1}), \varepsilon).$$
By the symmetry argument, we have that for every $g \in S$,
$$\rho_n(g^{-1})M_{v(g)} \subseteq \mathcal{U}^+(\rho_n(g^{-1}), \varepsilon).$$
Observe that, given an edge $e$ going from $v(g_1)$ to $v(g_2)$, its label must be $g_2$. 
The previous inclusion implies
$$\overline{\rho_n(\pi(e)^{-1})M_{v(g_2)}} \subseteq \mathcal{U}^+(\rho_n(g_2^{-1}), \varepsilon) \subseteq M_{v(g_1)},$$
which is the invariance condition. This concludes the theorem.
\end{proof}

\subsection{Application to the slice with constant triple ratio}

We finish the paper by reaping the benefits of having a strongly separated sequence given by Theorem \ref{thm:general+}.
For simplicity, we will restrict to a particular slice of convex projective structures of independent interest, for which we can use the results from the previous section.
More precisely, we will consider $\mathbb{RP} (X_1, X_2)$, the family of holonomies of convex projective structures for which the triple ratio invariants are constant and equal to $(X_1, X_2)$.

Let $(\mathcal{V}, \mathcal{E}, \pi)$ be the coding given by Example \ref{ex:codings}. Note that for this example, $X_\mathcal{G}$ is conjugated to a shift of finite type $X_A$, with the matrix 
      $$
 A =
        \begin{pmatrix}
            1 & 0 & 1 & 0 & 1 & 1 \\
            1 & 1 & 0 & 1 & 0 & 1 \\
            0 & 1 & 1 & 1 & 1 & 0 \\
            0 & 1 & 1 & 1 & 1 & 0 \\
            1 & 0 & 1 & 0 & 1 & 1 \\
            1 & 1 & 0 & 1 & 0 & 1
        \end{pmatrix},
      $$
 where the columns, in order, represent $a, b, c, a^{-1}, b^{-1}, c^{-1}$ respectively. 

The following is the combination of Theorem \ref{thm:divergingpants} and Theorem \ref{thm:general+}:

\begin{thm}
  \label{thm:sl3}
 Let $\varphi \in \mathfrak{a}^*$ and $(\vec{Z_t})_{t \in \mathbb{R}_{\geq 0}}, (\vec{W_t})_{t \in \mathbb{R}_{\geq 0}}$ be a path of vectors such that $\vec{Z_t},\vec{W_t} \to \infty$.
 Then, given $(\rho_t) \in \mathbb{RP}(X_1, X_2)^{\mathbb{R}_{\geq 0}}$ a path of holonomies of convex projective structures with edge invariants $\vec{W_t},\vec{Z_t}$, and assuming $\varphi$ is positive on their limit cone, we get that
   \begin{enumerate}
      \item Let $v_i (t) = (\log (W^i_t Z^{i-1}_t)) \omega_1 + \log (Z^i_t W^{i-1}_t) \omega_2$, $w_i (t) = (\log (Z^i_t W^{i-1}_t)) \omega_1 + (\log (W^i_t Z^{i-1}_t)) \omega_2$, we get that
         $$\lim_{t \to \infty }\frac{h_{\rho_t} (\varphi)}{s(t)} =1,$$
 where $s(t)$ is a positive solution to the equation
         $$\prod_{i=1}^3 (1-e^{-s (t) \varphi (v_i (t))}) + \prod_{i=1}^3 (1-e^{-s (t) \varphi (w_i (t))}) + \sum_{i=1}^3 e^{-s (\varphi (v_i + w_i))} =1.$$
      \item Let $\rho_t \in \mathbb{RP} (X_1, X_2)$, and $\rho_t' \in \mathbb{RP} (X_1', X_2')$ be a path of representations in the constant triple ratio slice, and defined by the same family of cross ratios given by $\vec{Z_t},\vec{W_t}$, then
      $$d_{Th}^\varphi (\rho_t,\rho_t') \to_{t \to \infty} 0.$$
      \item If $\max \{\varphi (v_i (t)),\varphi (w_i (t))\}_{i=1,2,3}/\min \{\varphi (v_i (t)),\varphi (w_i (t))\}_{i=1,2,3}$ is bounded above and below for every $t$, then the Thurston asymmetric metric defined by $\varphi$ is finite.
   \end{enumerate}
\end{thm}

\begin{proof}
\begin{enumerate}
    \item 
 Recall that the automaton defined by the symmetric set of generators $a^{\pm 1}, b^{\pm 1}, c^{\pm 1}$ is conjugated to the shift of finite type for the matrix $A$. 

 By Theorem \ref{thm:general+}, we have that $$\lim_{t \to \infty} \frac{h_{\rho_t}(\varphi)}{h_{\rho_t}^\kappa(\varphi)} = 1,$$
 where $h_{\rho_t}^\kappa(\varphi)$ is the only solution to $P(-sC_\varphi) = 0$, for $C$ the locally constant potential given by the Cartan projection of the first letter of the coding. 
    
 Notice that under the strong separation conditions, $\varphi (\kappa (\rho_n (g)))$ is asymptotic to $\varphi (\lambda (\rho_n (g)))$.
 Particularly, we get that $\lim_{t \to \infty} h_{\rho_t}^\kappa/h_{\rho_t}^\lambda = 1$.
 Using Lemma \ref{lemma:locallyconstant}, we get that $h_{\rho_t}^\lambda$ is the unique number for which the matrix:

 $$\mathcal{L}_s = 
 \begin{pmatrix}
 e^{- s \varphi (w_1 (t))} & 0 & e^{- s \varphi (w_1 (t))} & 0 & e^{- s \varphi (w_1 (t))} & e^{- s \varphi (w_1 (t))}\\
 e^{-s \varphi (w_2 (t))} & e^{-s \varphi (w_2 (t))} & 0 & e^{-s \varphi (w_2 (t))} & 0 & e^{-s \varphi (w_2 (t))}\\
   0 & e^{-s \varphi (w_3 (t))} & e^{-s \varphi (w_3 (t))} & e^{-s \varphi (w_3 (t))} & e^{-s \varphi (w_3 (t))} & 0\\
   0 & e^{-s \varphi (v_1 (t))} & e^{-s \varphi (v_1 (t))} & e^{-s \varphi (v_1 (t))} & e^{-s \varphi (v_1 (t))} & 0\\
 e^{-s \varphi (v_2 (t))} & 0 & e^{-s \varphi (v_2 (t))} & 0 & e^{-s \varphi (v_2 (t))} & e^{-s \varphi (v_2 (t))}\\
 e^{-s \varphi (v_3 (t))} & e^{-s \varphi (v_3 (t))}  & 0 & e^{-s \varphi (v_3 (t))}  & 0 & e^{-s \varphi (v_3 (t))} .
 \end{pmatrix}
 $$
 An algebraic calculation (maybe done by computer) produces an expression for the characteristic polynomial of this matrix.
 Assuming one is a solution, one gets the equation in the first item.
 
      \item Again, since $\varphi (\kappa (\rho_n (g)))$ is asymptotic to $\varphi (\lambda (\rho_n (g)))$, it is enough to estimate the former.
 On the other hand, Lemma \ref{lemma:eigenvalues} shows that the asymptotics for the Jordan projection of $a,b,c$ at constant triple ratio are independent of $(X_1, X_2)$.
 Particularly, the asymptotics for the length function are independent of which slice you are in.
      
   \item This is a direct consequence of the second part of Theorem \ref{thm:general+}.
\end{enumerate}
\end{proof}

\begin{rmk}
 Notice that 
   $$\frac{2 \log 2}{\max_i \{ \varphi (v_i (t)),\varphi (w_i (t)) \}} \leq s(t) \leq \frac{2 \log 2}{\min_i \{ \varphi (v_i (t)),\varphi (w_i (t)) \}},$$
 for every $t$.
 This follows from the proof by bounding $\mathcal{L}_s$ above and below the positive matrix $\mathcal{L}_{-s}$ by positive matrices depending only on the min/max of $\varphi (v_i)$ and $\varphi (w_i)$.
\end{rmk}

%% file: sections/finite_type.tex
\section{Subshifts of finite type} 
\label{app:subshifts}
In this appendix we review some theory of subshifts of finite type. These will be used for explicit computations in Section \ref{sec:sl3}. For more information on this topic one can check \cite{parry1990zeta}.

Let $\mathcal{A} = \{a_1, \ldots a_k\}$ be an alphabet consisting of $k$-symbols, and consider $A = (A(i,j))_{i,j \in \{1,\ldots, k\}}$ a \emph{grammar matrix} composed of $0$'s and $1$'s.
We define the \emph{(one-sided) shift space} as
$$X_A = \{(x_n)_{n\geq 0} \in \mathcal{A}^\mathbb{N} : A(x_n,x_{n+1}) =1,\: \forall n \}.$$
This set can be recovered from a directed graph $\mathcal{G}_A$ with vertices $\mathcal{A}$, and a directed edge connecting symbols where $A (x_i,x_{i+1}) = 1$.

If we denote by $T: X_A \to X_A$ the function such that $T((x_n)) = (x_{n+1})$, we get that $(X_A, T)$ is a dynamical system, we call this system a \emph{shift of finite type}.

We will always assume that the associated graph $\mathcal{G}_A$ is irreducible and aperiodic.
Under these conditions, $(X_A, T)$ is topologically mixing and admits plenty of invariant measures. The pressure functional $P(\phi)$, defined in Definition \ref{def:pressure}, 
can also be defined for shifts of finite type via the same formula.

It turns out that one can recover the pressure and the equilibrium measure as the eigenvalue (respectively, eigenvector) of the \emph{Ruelle operator}.
Let $C^\alpha (X_A)$ be the space of $\alpha-$H\"older functions from $X_A$ to $\RR$. Given $\phi \in C^\alpha (X_A)$, we define the Ruelle operator $\mathcal{L}_\phi : C^\alpha (X_A) \rightarrow C^\alpha (X_A)$ via
$$\mathcal{L}_\phi (f) (\underline{x}) = \sum_{\underline{y}\in T^{-1}(\underline{x})} e^{\phi (\underline{y})} f(\underline{y}).$$
This operator is positive in the sense that it preserves the cone of positive H\"older functions.
Moreover, we have the following theorem:

\begin{thm}[{\cite[Theorem 3.5]{parry1990zeta}}]
   \label{thm:parry}
 Given $\phi \in C^\alpha (X_A)$, the operator $\mathcal{L}_\phi$ has a simple eigenvalue with maximal modulus $e^{P(\phi)}$.
\end{thm}

A potential $\phi \in C(X_A)$ is called \emph{locally constant} if it depends only on the first coordinate of a sequence $\underline{x} \in X_A$. Such a potential defines a distance on $\mathcal{G}_A$ by declaring $d_\phi (e_{ij}) = \phi (i)$.
In these conditions, any eigenfunction for $\mathcal{L}_\phi$ will again be locally constant.
As a consequence, we have the following lemma:

\begin{lem}
\label{lemma:locallyconstant}
Let $\phi \in C^{\alpha}(X_A)$ be a locally constant function. We define the \emph{weighted matrix} $A_\phi$ as the matrix obtained by replacing $a_{ij}$ in $A$ by $e^{\phi(i)}$. Then the spectral radius of $A_\phi$ is $e^{P(\phi)}$.
\end{lem}
\begin{proof}
Locally constant functions form a $\mathcal{L}_\phi$-invariant subspace on which the operator acts as $A_\phi$. Since $A_\phi$ is a non-negative irreducible matrix, its leading eigenvalue coincides with the spectral radius $e^{P(\phi)}$ given by Theorem~\ref{thm:parry}.
\end{proof}